\documentclass[12pt]{amsart}
\usepackage{amssymb}
\usepackage{amsfonts}
\usepackage{amsmath,latexsym,amssymb,amsfonts,amsbsy, amsthm, enumitem}
\usepackage[usenames]{color}
\usepackage{xspace,colortbl}
\usepackage{mathrsfs}
\usepackage[colorlinks,linkcolor=blue,citecolor=blue]{hyperref}
\usepackage[titletoc]{appendix}
\allowdisplaybreaks

\newcommand{\beq}{\begin{equation}}
\newcommand{\eeq}{\end{equation}}
\newcommand{\ben}{\begin{eqnarray}}
\newcommand{\een}{\end{eqnarray}}
\newcommand{\beno}{\begin{eqnarray*}}
\newcommand{\eeno}{\end{eqnarray*}}
\newcommand{\R}{\mathbb{R}}

\newtheorem{thm}{Theorem}[section]
\newtheorem{defi}[thm]{Definition}

\newtheorem{lem}[thm]{Lemma}
\newtheorem{prop}[thm]{Proposition}
\newtheorem{coro}[thm]{Corollary}
\newtheorem{rmk}[thm]{Remark}

\newtheorem{conj}[thm]{Conjecture}

\begin{document}

\title[Blow up analysis for Keller-Segel]{Blow up analysis for Keller-Segel system}
\author[H. Chen, J.-M. Li and K. Wang]{Hua Chen, Jian-Meng Li and Kelei Wang}
\thanks{School of Mathematics and Statistics, Wuhan University, Wuhan 430072, China.}
\thanks{ Email: chenhua@whu.edu.cn (H. Chen), lijianmeng@whu.edu.cn (J.  Li), wangkelei@whu.edu.cn (K. Wang). }

\thanks{This work is supported by  National Key R\&D Program of China (No. 2022YFA1005602) and the National Natural Science Foundation of China (No. 12131017, No. 12221001 and No. 12425108). We are grateful to the referees for their careful reading and many valuable suggestions.}
\date{\today}

\begin{abstract}
In this paper we develop a blow up theory for the parabolic-elliptic Keller-Segel system, which can be viewed as a parabolic counterpart to   the Liouville equation. This theory is   applied to  the study of   first time singularities, ancient solutions and entire solutions, leading to a description of the blow-up limit in the first problem, and  the large scale structure in the other  two problems.

\end{abstract}
\keywords{Keller-Segel system; blow up analysis; first time singularity; entire solution.}

\subjclass[2020]{35K58, 35B44, 35B33.}

\maketitle
\renewcommand{\theequation}{\thesection.\arabic{equation}}
\setcounter{equation}{0}
%%%%%%%%%%%%%%%%%%%%%%%%%%%%%%%%%%%%%%%%%%%%%%

\tableofcontents

\section{Introduction}\label{sec introduction}
\setcounter{equation}{0}

\subsection{Blow up analysis}
In this paper we generalize the blow up analysis of Brezis-Merle \cite{Brezis-Merle} and Li-Shafrir \cite{Li-Shafrir} for the Liouville equation
\begin{equation}\label{Liouville eqn}
	-\Delta u=e^u
\end{equation}
to a parabolic setting, that is, for the Keller-Segel system (see Keller-Segel \cite{keller1971model})
\begin{equation}\label{eqn}
\left\{
  \begin{array}{ll}
    u_t=\Delta u-\mbox{div}(u\nabla v),  \\
    -\Delta v=u.
 \end{array}
\right.
\end{equation}
We will also  apply this theory to the study of first time singularities in \eqref{eqn} and the large scale structure of ancient and entire solutions of \eqref{eqn}.

Throughout this paper the spatial dimension is $2$.
Our first main result is about the convergence and blow up behavior for sequences of solutions to \eqref{eqn}.
\begin{thm}\label{main result}
Suppose $u_i$ is a sequence of smooth, positive  solutions of the Keller-Segel system \eqref{eqn} in the unit parabolic cylinder $Q_1:=B_1\times(-1,1)\subset\R^2\times\R$, satisfying
\begin{equation*}
	\sup_{-1<t<1}\int_{B_1}u_i(x,t)dx\leq M \quad \mbox{for some constant} ~~ M.
\end{equation*}
Then the followings hold.
	\begin{enumerate}
		\item  There exists a family of Radon measures $\mu_t$ on $B_1$, $t\in(-1,1)$, such that after passing to a subsequence of $i$, for any $t\in(-1,1)$,
		\[ u_i(x,t)dx\rightharpoonup \mu_t  \quad \text{weakly as Radon measures};\]
		\item $\mu_t$ is continuous in $t$ with respect to the weak topology;
		\item  for any $t\in(-1,1)$, there exist $N(t)$ points $q_j(t)$, where $0\leq N(t)\leq M/(8\pi)$,  and $0\leq \rho(t)\in L^1(B_1)$ such that
		\begin{equation}\label{multiplicity one}
			\mu_t=\sum_{j=1}^{N(t)}8\pi\delta_{q_j(t)}+\rho(x,t)dx.
		\end{equation}
		\item the blow up locus $\Sigma:=\cup_{t\in(-1,1)}\cup_{j=1}^{N(t)}\{(q_j(t),t)\}$ is relatively closed in $Q_1$;
		\item $\rho\in C^\infty(Q_1\setminus \Sigma)$ and satisfies \eqref{eqn} in this open set, where
\[\nabla v(x,t)=-4\sum_{j=1}^{N(t)}\frac{x-q_j(t)}{|x-q_j(t)|^2}-\frac{1}{2\pi}\int_{B_1}\frac{x-y}{|x-y|^2}\rho(y,t)dy;\]
		\item if $\rho$ is smooth in $Q_1$, then $N(t)\equiv N$ for some $N\in\mathbb{N}$, $q_j(t)\in C^1(-1,1)$ for each $j=1$, $\dots$, $N$ and
			\begin{equation}\label{dynamical law}
			q_j^\prime(t)= 4\sum\limits_{k\neq j}\dfrac{q_k(t)-q_j(t)}{|q_k(t)-q_j(t)|^2}+\frac{1}{2\pi}\int_{B_1}\frac{y-q_j(t)}{|y-q_j(t)|^2}\rho(y,t)dy.
		\end{equation}
	\end{enumerate}
\end{thm}

\begin{rmk}\label{rmk 1.1}
\begin{enumerate}
	\item Eqn. \eqref{multiplicity one} says that each Dirac measure in the singular limit has mass $8\pi$. Near this blow up point, $u_i$ should look like a scaled bubble, that is,
	\[u_i(x,t)\sim \frac{8\lambda_{j,i}(t)^2}{(\lambda_{j,i}(t)^2+|x-q_{j,i}(t)|^2)^2}, \quad \text{where} ~~ \lambda_{j,i}(t)\to0, ~ q_{j,i}(t)\to q_j(t),\]
	 see Theorem \ref{thm construction of entire sol} and Conjecture \ref{conjecture on entire solution} below. This essentially means that there is only one bubble at each blow up point, or using the terminology of bubbling analysis, the blow up is isolated and simple.
	
	 It is also worth noticing that although these Dirac measures are separated for each $t<1$, it is not claimed that they cannot converge to the same one as $t\to1$. In fact, this is the picture for the first time singularity in Keller-Segel system as constructed in Seki-Sugiyama-Vel\'{a}zquez \cite{Seki-Sugiyama-Velazquez}.  In this sense, our result corresponds to the quantization phenomena  for blow ups of Liouville equations, where the mass of  Dirac measures in the singular limit is $8\pi N$ for some $N\in\mathbb{N}$ (Li-Shafrir \cite{Li-Shafrir})  and $N$ could be larger than $1$ (Chen \cite{ChenX}).

\item The number of atoms in $\mu_t$, $N(t)$, could be a non-constant function of $t$, see Remark \ref{rmk on N(t)} for more discussions.

\item The first term in the right hand side of \eqref{dynamical law} describes the interaction between different Dirac measures. If $N=1$,  \eqref{dynamical law} should be understood as without this term.
	\item If the diffusion part $\rho(x,t)dx$ does not appear, then \eqref{dynamical law} becomes an ODE system and it is the gradient flow of the renomarlized energy
	\[W(q_1,\cdots, q_N):=4\sum_{j\neq k}\log|q_j-q_k|.\]
It arises as a renomalization of the free energy
	\[\mathcal{F}(u):=\int u(x)\log u(x)dx+\frac{1}{4\pi}\int\int \log|x-y|u(x)u(y)dxdy.\]
	Roughly speaking, if  $u$ is close to $8\pi \sum_{j=1}^{N}\delta_{q_j}$, then after subtracting the self-interaction terms from $\mathcal{F}$ (which could be very large), only those terms describing interaction between different Dirac measures are left, which is $W$, see related  discussions in \cite[Chapter 2]{Suzuki1}.
	%It is well known that \eqref{eqn} is the gradient flow of $\mathcal{F}$ in the space of measures, so it is not a coincidence that \eqref{dynamical law} is the gradient flow of a renomarlized energy.

  For the singular limit of a modified Keller-Segel system, a similar point dynamics was also established in Vel\'{a}zquez \cite{Velazquez-point1, Velazquez-point2}.
\end{enumerate}
\end{rmk}

The proof of Theorem \ref{main result} uses several tools such as symmetrization of test functions and $\varepsilon$-regularity theorems. The use of $\varepsilon$-regularity theorems is standard, just as in the study of blow up phenomena for many other PDE problems. The symmetrization of test function technique is mainly used to   calculate the equation for the first and second momentum, which have also been used by many people in the study of Keller-Segel system. However, because our setting is local in nature, which is different from most literature on Keller-Segel system,  a suitable \emph{localization of these calculations} is necessary, see Lemma \ref{lem localization} below for the definition of this localization procedure.  In particular, the proof of \eqref{multiplicity one} relies strongly on a localized calculation of the second momentum. (It is also performed in a limiting form, see
Section \ref{sec multiplicity one} for  details.) The derivation of \eqref{dynamical law} also uses a localized calculation of the first momentum,  see
Section \ref{sec limiting ODE}.

 In the literature, there are various notions of weak solutions about \eqref{eqn}, see e.g.  Biler \cite{Biler-book}, Suzuki \cite[Chapter 13]{Suzuki1} and Luckhaus-Sugiyama-Vel\'{a}zquez \cite{Luckhaus-Sugiyama-Velazquez}. These notions are related to the dynamical law \eqref{dynamical law}, but we will not use them in this paper. Instead, to derive \eqref{dynamical law}, we rely solely on explicit formulas derived from the convergence of sequences of smooth solutions.

\subsection{First time singularity} Theorem \ref{main result} provides a convenient setting for the analysis of blow up behavior in Keller-Segel system. Our first application of Theorem \ref{main result} is on the analysis of  first time singularities in the Keller-Segel system \eqref{eqn}.

By standard parabolic theory, under suitable assumptions, there exists a local solution to the Cauchy problem of \eqref{eqn} on $\R^2$. The solution may not exist globally in time. For example, when the total mass
is larger than the critical one, $8\pi$, the solution must blow up in finite time, see Dolbeault-Perthame \cite{Dolbeault-Perthame}.
Then it is  natural to analyse the blow up behavior when the solution  blows up at the first time. For the Keller-Segel system \eqref{eqn}, blow up of solutions is caused by \emph{aggregation}, that is,   \emph{concentration of mass}. The following is \cite[Theorem 1.1]{Suzuki1}
or \cite[Theorem 1.1]{Suzuki3}. (Similar results also hold for initial-boundary value problems.)
\begin{thm}[Suzuki]\label{thm mass concentration}
	Suppose $u$ blows up at finite time $T$. Then as $t\to T$,
	\[ u(x,t)dx \rightharpoonup u_T(x)dx+ \sum_{a_i}m_i\delta_{a_i}\]
	weakly as Radon measures,  where $\{a_i\}$, the set of blow up points, is a finite set of $\R^2$ and $m_i\geq 8\pi$,
	$u_T\in L^1(\R^2)\cap C(\R^2\setminus\{a_i\})$ is a nonnegative function.
\end{thm}
For the construction of finite time blow up solutions, see Herrero-Vel\'{a}zquez \cite{herrero1997blow},  Vel\'{a}zquez \cite{Velazquez2002}, Rapha\"{e}l-Schweyer \cite{Raphael-S}, Collot et. al. \cite{Collot-Ghoul-Masmoudi-Nguyen, Masmoudi2024} and Buseghin et. al. \cite{Davila2023} .

By Theorem \ref{thm mass concentration}, to study first time singularites, we can work in the following local setting:
\begin{description}
 \item [ (H1)]  $u \in C^\infty(\overline{Q_1^-}\setminus\{(0,0)\})$ (here $Q_1^-:=B_1\times(-1,0)$ is the unit backward parabolic cylinder), $u>0$ and
  \begin{equation}\label{mass bound}
    \sup_{t\in(-1,0)}\int_{B_1}u(x,t)dx\leq M;
  \end{equation}
 \item [(H2)]  $u$ satisfies \eqref{eqn} in $Q_1^-$, where $\nabla v$ is given by
\begin{equation}\label{representation for v}
\nabla v(x,t)=-\frac{1}{2\pi}\int_{B_1}\frac{x-y}{|x-y|^2}u(y,t)dt, \quad \forall (x,t)\in Q_1^-;
\end{equation}
 \item [(H3)]  there exists a nonnegative function $u_0\in L^1(B_1)$ and a positive constant $m$ such that as $t\to 0^-$,
  \[ u(x,t)dx\rightharpoonup u_0(x)dx+m\delta_0\]
  weakly as Radon measures.
\end{description}
Under these hypothesis, we will examine the behavior of $u$ near the blow up point $(0,0)$. For this purpose, observe that the Keller-Segel system \eqref{eqn} is invariant under the scaling
\begin{equation}\label{blow-up sequence}
	u^\lambda(x,t):=\lambda^2 u(\lambda x,\lambda^2 t), \quad \nabla v^\lambda(x,t)=\lambda\nabla v(\lambda x,\lambda^2 t), \quad \forall \lambda>0.
\end{equation}
Moreover,  because the spatial dimension is $2$, the $L^1$ norm of $u$ is invariant under this scaling.
The small scale structure of $u$ near $(0,0)$ can be revealed by examining the convergence of  $u^\lambda$ as $\lambda \to 0$. This is the blow-up procedure.

By applying Theorem \ref{main result}, we get the following result about the blow-up sequences $u^\lambda$.
\begin{thm}\label{thm first time singularity}
	Under hypothesis {\bf(H1)}-{\bf(H3)}, the followings hold.
 \begin{description}
 	\item [(i) Quantization] There exists an $N\in\mathbb{N}$ such that $m=8\pi N$.
 	\item [(ii) Blow-up limit and Self-similarity] For any sequence $\lambda_i\to0$, there exists a subsequence (not relabelling) and $N$ distinct points $p_j\in \R^2$ such that, for any $t<0$,
 	\[u^{\lambda_i}(x,t)dx \rightharpoonup 8\pi\sum_{j=1}^{N}\delta_{\sqrt{-t}p_j} \quad \mbox{weakly as Radon measures}.\]
 	\item [(iii) Renormalized energy] If $N=1$, then $p_1= 0$. If $N\geq 2$, then the $N$-tuple $(p_1,\cdots, p_N)$ is a critical point of the renormalized energy
 	\begin{equation}\label{renormalized energy}
 		\mathcal{W}(p_1,\cdots, p_N):= -\frac{1}{4}\sum_{j=1}^{N}|p_j|^2+4\sum_{1\leq j\neq k \leq N}\log|p_j-p_k|.
 	\end{equation}	
 \end{description}
\end{thm}

\begin{rmk}
\begin{enumerate}
	\item In the converse direction, finite time blow up solutions with
	asymptotic behavior described as in this theorem  have been constructed in Seki-Sugiyama-Vel\'{a}zquez \cite{Seki-Sugiyama-Velazquez} by using the matched asymptotics method, see also the recent work of Collot et. al. \cite{Masmoudi2024}.
	\item If $N=2$, critical points of $\mathcal{W}$ must be a symmetric pair, that is,  $(-p,p)$ for some $p\in\R^2$.
	\item  In many parabolic equations, self-similarity of blow-up limits is established with the help of a monotonicity formula. However, in the above theorem, this is proved by using the facts that, the ODE system \eqref{dynamical law} (if there is no diffusion part) is the gradient flow of $W$, and the critical energy levels of $\mathcal{W}$ are discrete.
	\item In the renomarlized energy $\mathcal{W}$, the first  term  comes from a self-similar transformation.  As explained in Remark \ref{rmk 1.1}, the second term comes from a renormalization of the free energy $\mathcal{F}$.

	\item Because blow-up limits are obtained by a compactness argument, we do not know if the $N$-tuple $(p_1,\cdots, p_N)$ is unique.
	\end{enumerate}
\end{rmk}

Most of these results are not new, cf. \cite[Theorem 14.2]{Suzuki1} or \cite[Chapter 1]{Suzuki3}. One main difference is the method of the proof. In \cite{Suzuki1} and \cite{Suzuki3},
Suzuki uses the self-similar transformation, i.e. by considering similarity variables
\[ y=\frac{x}{\sqrt{-t}}, \quad s=-\log(-t)\]
and then taking the transformation
\[z(y,s):= |t|u(x,t), \quad w(y,s):=v(x,t), \]
one gets  the system
\begin{equation}\label{self-similar eqn}
		\left\{
		\begin{array}{ll}
		z_s-\Delta z= -\mbox{div}\left[z\left(\nabla w+\frac{y}{2}\right)\right],\\
			-\Delta w=z.
		\end{array}
		\right.
	\end{equation}
Then the proof is reduced to analyse the large time behavior of $(z,w)$. Our proof uses instead the blow up method, by
considering the rescalings in \eqref{blow-up sequence}
and then analyse its convergence as $\lambda\to 0$. These two convergence analysis are essentially equivalent, but in some cases the later one
provides more information. For example, by examining the scalings $(u^\lambda, v^\lambda)$ at $t=0$, we see that the trace of $u$ at $t=0$, $u_0$,
does not enter the singularity formation mechanism on the mass level. More precisely, Theorem \ref{thm first time singularity} implies that
\begin{coro}
 In $Q_{1/2}^-$,
  \[ u(x,t)=o\left( \frac{1}{(|x|+\sqrt{-t})^2}\right).\]
 In particular,
 \[ u_0(x)=o\left(\frac{1}{|x|^2}\right).\]
\end{coro}
When the solution is radially symmetric, a rather precise asymptotic expansion of $u_0$ near the origin has been given in  Herrero-Vel\'{a}zquez \cite{Herrero-Velazquez} and Mizoguchi  \cite{Mizoguchi2022}.
But at present  it is still not known what further regularity on $u_0$ can be obtained in  the above  general setting.
Such a knowledge will be helpful in defining the continuation after the blow up time (cf. Dolbeault-Schmeiser \cite{Dolbeault-Schmeiser}).

\subsection{Ancient solutions}

\begin{defi}
If $u$, $\nabla v$ are smooth in $\R^2\times(-\infty,0]$ and satisfy \eqref{eqn}, where the second equation is understood as
\begin{equation}\label{representation for v, 2}
	\nabla v(x,t)=-\frac{1}{2\pi}\int_{\R^2}\frac{x-y}{|x-y|^2}u(y,t)dt, \quad \forall (x,t)\in \R^2\times(-\infty,0],
\end{equation}
then it is called
an ancient solution of \eqref{eqn}.
\end{defi}
Ancient solutions play an important role in the analysis of blow up phenomena for many nonlinear parabolic equations.

Similar to the blow up analysis used in the proof of Theorem \ref{thm first time singularity}, we can analyse the large scale structure of ancient solutions by the blow-down analysis, that is, to consider the convergence of the  sequence
defined in \eqref{blow-up sequence}, but now with
 $\lambda\to+\infty$.
\begin{thm}\label{thm ancient sol}
Assume that $(u,v)$ is an ancient solution, satisfying
 \begin{equation*}\label{mass bound ancient}
	\sup_{t\leq 0}\int_{\R^2}u(x,t)dx<+\infty.
\end{equation*}
Then the followings hold.
 \begin{description}
	\item [(i) Quantization] There exists an $N\in\mathbb{N}$ such that
	\[\int_{\R^2}u(x,t)dx\equiv 8\pi N.\]
	\item [(ii) Blow-down limit] For any sequence $\lambda_i\to +\infty$, there exists a subsequence (not relabelling) and $N$ distinct points $p_j\in \R^2$ such that, for any $t<0$,
	\[u^{\lambda_i}(x,t)dx \rightharpoonup 8\pi\sum_{j=1}^{N}\delta_{\sqrt{-t}p_j} \quad \mbox{weakly as Radon measures}.\]
	\item [(iii) Renormalized energy] If $N=1$, then $p_1=0$. If $N\geq 2$, then the $N$-tuple $(p_1,\cdots, p_N)$ is a critical point of the renormalized energy
	$	\mathcal{W}$.
\end{description}
\end{thm}
As can be seen, the statement is almost the same with Theorem \ref{thm first time singularity}. In fact, the proof is  also the same, which is still mainly an application of Theorem \ref{main result}. One only difference is that, because now it is a backward problem, when dealing with the analysis of gradient flows $p(s)$ of the renormalized energy $\mathcal{W}$, we need to consider $\lim_{s\to-\infty}p(s)$ ($\alpha$-limit sets) instead of $\lim_{s\to+\infty}p(s)$ ($\omega$-limit sets).

\subsection{Entire solutions}

Finally, we study the large scale structure of entire solutions.
\begin{defi}
	If $u$,$\nabla v$ are smooth in $\R^2\times \R$ and satisfy \eqref{eqn}, where the second equation is understood as \eqref{representation for v, 2}, then it is called
	an entire solution of \eqref{eqn}.
\end{defi}
Entire solutions appear as a suitable rescalings around   blow up points, see Theorem \ref{thm construction of entire sol}. Therefore it is the micro-model of singularity formations at blow up points.

The large scale structure of entire solutions is described in the following theorem.
\begin{thm}\label{thm entire sol}
	Assume that $(u,v)$ is an entire solution, satisfying the finite mass condition
	\begin{equation}\label{finite mass condition}
\sup_{t\in\R}\int_{\R^2}u(x,t)dx<+\infty.
\end{equation}
Then
	as $\lambda\to +\infty$,  for any $t$,
		\[u^{\lambda}(x,t)dx \rightharpoonup 8\pi\delta_{0} \quad \mbox{weakly as Radon measures}.\]
\end{thm}
For a related property on weak solutions, called weak Liouville property, see \cite[Lemma 1.4]{Suzuki3}.

Because the blow-down limit for entire solutions is time-independent, it is natural to conjecture that the original entire solution is also time-independent. Then the parabolic equation \eqref{eqn} is reduced to
\[
\left\{ \begin{aligned}
	&\Delta u-\mbox{div}\left(u\nabla v\right)=0, \\
	&-\Delta v=u.
\end{aligned} \right.
\]
This is essentially the Liouville equation \eqref{Liouville eqn}.
By the Liouville theorem of Chen-Li \cite{Chen-Li}, we thus have
\begin{conj}\label{conjecture on entire solution}
	Any nontrivial entire solution of \eqref{eqn} satisfying
	the finite mass condition \eqref{finite mass condition}
has the form
	\[
	\left\{ \begin{aligned}
		&u(x,t)\equiv \frac{8\lambda^2}{(\lambda^2+|x-\xi|^2)^2}, \\
		&v(x,t)\equiv -2\log\left(\lambda^2+|x-\xi|^2\right)+C,
	\end{aligned} \right.
	\]
for some $\lambda>0$, $\xi\in\mathbb{R}^2$ and $C\in\R$.
\end{conj}

{\bf Notations:} Throughout this paper, we use the following notations.
\begin{itemize}
  \item A standard cut-off function $\psi$ subject to  two bounded smooth domains $\Omega_1 \subset \Omega_2$ is a function
  satisfying $\psi\in C_0^\infty(\Omega_2)$, $0\leq \psi \leq 1$, $\psi\equiv 1$ in $\Omega_1$ and for any $k\geq 1$,
  \[ |\nabla^k\psi|\leq C_k \mbox{dist}(\Omega_1, \Omega_2^c)^{-k}.\]
  
  \item A point in space-time is denoted by $(x,t)$. The parabolic distance is
  \[\text{dist}_p((x,t),(y,s)):= \max\left\{|x-y|, |t-s|^{\frac{1}{2}}\right\}.\]
  
  \item A parabolic cylinder is $Q_r(x,t):=B_r(x)\times (t-r^2,t+r^2)$ and a backward parabolic cylinder is $Q_r^-(x,t):=B_r(x)\times (t-r^2,t)$.

  \item We use standard notation of parabolic H\"{o}lder spaces, e.g. $C^{\alpha, \alpha/2}$ denotes a function which is $\alpha$-H\"{o}lder in $x$ and $\alpha/2$-H\"{o}lder in $t$.
\end{itemize}

The organization of this paper is as follows. Section \ref{sec setting}-Section \ref{sec limiting ODE} is devoted to the blow up analysis, in which the proof of Theorem \ref{main result} will be given. We study first time singularities and prove Theorem \ref{thm first time singularity} in Section \ref{sec blow-up limits}. In Section \ref{sec entire solutions} we prove Theorem \ref{thm entire sol} on entire solutions and explain how to construct entire solutions by rescalings. Finally, we give a brief remark on boundary blow up points in Section \ref{sec boundary blow up}.

\section{Setting}\label{sec setting}
\setcounter{equation}{0}

We will work in the following local setting, which is a little more general than \eqref{eqn}:  $u,v\in C^\infty(Q_1)$, $u\geq 0$, and they  satisfy
\begin{equation}\label{eqn modified}
\left\{
  \begin{array}{ll}
    u_t=\Delta u-\mbox{div}\left[u\nabla (v+f)\right]+g,  \\
    -\Delta v=u
 \end{array}
\right.
\end{equation}
Here we assume
 \begin{itemize}
   \item  there exists a constant $M>0$ such that
    \begin{equation}\label{mass bound 2}
    \sup_{t\in(-1,1)}\int_{B_1}u(x,t)dx\leq M;
  \end{equation}
   \item  $\nabla f$ and $g$ are two given functions, where for some $\alpha\in(0,1)$, $\nabla f\in C^{\alpha,\alpha/2}(Q_1)$, while $g\in L^\infty(Q_1)$;
   \item   the second equation in \eqref{eqn modified} is understood as
\begin{equation}\label{definition of v in local setting}
  \nabla v(x,t)=-\frac{1}{2\pi}\int_{B_1}\frac{x-y}{|x-y|^2}u(y,t)dy, \quad \forall t\in(-1,1).
\end{equation}
 \end{itemize}
The assumption on $g$ can be weakened, for example,
by denoting $p:=2/(1-\alpha)$,  $g\in L^p(Q_1)$ is sufficient for most arguments in this paper. But that will lead to some technical issues, and we will not pursue it here.

\begin{rmk}
  The equation \eqref{self-similar eqn} in self-similar coordinates is a special case of \eqref{eqn modified}.
\end{rmk}

The form of Eq. \eqref{eqn modified} is invariant under localization in the following sense.  
\begin{lem}[Localization]\label{lem localization}
For any $\eta\in C_0^\infty(B_1)$, $\eta\geq 0$, define $\widetilde{u}:=u\eta$.  Then $\widetilde{u}$ still satisfies \eqref{eqn modified}, with $\nabla v$, $f$ and $g$ replaced respectively by
\[
\left\{
\begin{aligned}
	& \nabla \widetilde{v}(x,t):=-\frac{1}{2\pi}\int_{B_1}\frac{x-y}{|x-y|^2}\widetilde{u}(y,t)dy,\\
	& \nabla \widetilde{f}(x,t):=\nabla f(x,t)+\nabla v-\nabla \widetilde{v},\\
	&\widetilde{g}(x,t):=g(x,t)\eta(x)+u(x,t)\left[\nabla v(x,t)+\nabla f(x,t)\right]\cdot\nabla\eta(x)
	\\
	&\qquad \quad  -2\nabla u(x,t)\cdot\nabla\eta(x)-u(x,t)\Delta\eta(x).
\end{aligned}\right.
\]
\end{lem}
The proof  is a direct calculation using the definition. As will be clear from discussions in the remaining parts of this paper, it is very convenient that solutions can be localized in the above way, especially when studying concentration phenomena in \eqref{eqn}. Thus this localization procedure will be employed many times in the below.

\section{Two basic tools}\label{sec tools}
\setcounter{equation}{0}

In this section, $(u,v)$ denotes a
classical solution of \eqref{eqn modified} in $Q_1$,  which satisfies \eqref{mass bound 2}. Here we recall two tools, symmetrization of test functions and $\varepsilon$-regularity theorem. These two tools will be used a lot in this paper. 

The first tool is symmetrization of test functions.
For any $\psi \in C^2(\R^2)$, define
\begin{equation}\label{symmetrization of test fct}
  \Theta_{\psi}\left(x, y\right):=\frac{x-y}{|x-y|^2}\cdot\left[\nabla \psi(x)-\nabla \psi(y) \right].
\end{equation}
It belongs to $L^{\infty}(\mathbb{R}^2\times \mathbb{R}^2)$, with the estimate
\begin{equation}\label{symmetrization fct}
  \|\Theta_{\psi}\|_{L^{\infty}(\mathbb{R}^2\times \mathbb{R}^2)}\leq C\|\psi\|_{C^2(\R^2)}.
\end{equation}

This symmetrized function $\Theta_\psi$ appears in the following estimate.
\begin{lem}\label{lem symmetrization}
For any $\psi\in C_0^2(B_1)$,
\begin{eqnarray}\label{symmetrization identity}
	\frac{d}{dt}\int_{B_1}u\psi &=& \int_{B_1}u\Delta\psi+\int_{B_1}u\nabla f\cdot \nabla \psi+\int_{B_1}g\psi \\
	&-&\frac{1}{4\pi}\int_{B_1}\int_{B_1}\Theta_\psi(x,y)u(x,t)u(y,t)dxdy. \nonumber \qedhere
\end{eqnarray}	
	As a consequence,
there exists a universal constant $C$ such that
\begin{eqnarray}\label{symmetrization estimate}
\left|\frac{d}{d t} \int_{B_1} u(x,t) \psi(x) d x\right| &\leq & C \left(M^2\|\nabla^2\psi\|_{L^\infty(B_1)}+M\|\Delta \psi\|_{L^\infty(B_1)}\right) \nonumber\\
&&+CM\|\nabla f\|_{L^\infty(Q_1)}\|\nabla\psi\|_{L^\infty(B_1)}\\
&&+C\|g\|_{L^\infty(Q_1)}\|\psi\|_{L^\infty(B_1)}, \nonumber
 \end{eqnarray}
where $M$ is the mass bound in \eqref{mass bound 2}.
\end{lem}
\begin{proof}
  The proof  is similar to   the standard Keller-Segel system case, see e.g. \cite[Lemma 5.1]{Suzuki1}.
\end{proof}

The second tool is $\varepsilon$-regularity theorem.
\begin{thm}\label{thm ep regularity}
There exist two   small constants $0<\varepsilon_\ast\ll \theta_\ast \ll 1$ (depending only on the constant $M$ in \eqref{mass bound 2},  $\|\nabla f\|_{L^\infty(Q_1)}$ and $\|g\|_{L^\infty(Q_1)}$) such that, if 
\begin{equation}\label{small mass condition}
\int_{B_1(0)}u(x,0)dx\leq \varepsilon_\ast,
\end{equation}
then
\begin{equation}
  \|u\|_{C^{1+\alpha,(1+\alpha)/2}(Q_{\theta_\ast})}\leq C.
\end{equation}
\end{thm}
The proof is still similar to the one for the standard Keller-Segel system. It will be given in Appendix \ref{sec appendix}.

\section{Convergence of solutions}\label{sec convergence theory}
\setcounter{equation}{0}

In this section we study the convergence of sequences of smooth solutions to \eqref{eqn modified}.

From here until Section \ref{sec limiting ODE}, we assume $u_i$ is a sequence of smooth solutions of  \eqref{eqn modified} in $Q_1$,   satisfying:
\begin{enumerate}
  \item there exists a constant $M$ such that \begin{equation}\label{local uniform mass bound}
  \sup_{-1<t<1}\int_{B_1}u_i(x,t)dx\leq M;
\end{equation}
  \item $\nabla f_i$  are uniformly bounded in $C^{\alpha,\alpha/2}(Q_1)$, and  $\nabla f_i \rightarrow\nabla f$ uniformly in $Q_1$;

  \item   $g_i$ are uniformly bounded in $L^\infty(Q_1)$, and $g_i$ converges to $g$ $\ast$-weakly in $L^\infty(Q_1)$.
\end{enumerate}
Under these assumptions, we have
\begin{thm}\label{thm convergence}
	\begin{enumerate}[label=(\Alph*)]
		\item\label{item:limit} {\bf [Weak limit]}  There exists a family of Radon measures $\mu_t$ on $B_1$, $t\in(-1,1)$, such that after passing to a subsequence of $i$,
\begin{equation}\label{3.2}
u_i(x,t)dx\rightharpoonup \mu_t \quad \text{weakly as Radon measures}, \quad \forall t\in(-1,1);
\end{equation}
\item\label{item:continuity} {\bf [Weak continuity] }$\mu_t$ is continuous in $t$ with respect to the weak topology;
  \item\label{item:support} {\bf  [Decomposition of $\mu_t$] } for any $t\in(-1,1)$, there exist $N(t)$ points $q_j(t)$, $\varepsilon_\ast \leq m_i(t) \leq M$ and $0\leq \rho(t)\in L^1(B_1)$ such that
\[ \mu_t=\sum_{j=1}^{N(t)}m_j(t)\delta_{q_j(t)}+\rho(x,t)dx.\]
\item\label{item:blow up locus} {\bf  [Blow up locus] } the blow up locus $\Sigma:=\cup_{t\in(-1,1)}\cup_{j=1}^{N(t)}\{(q_j(t),t)\}$ is relatively closed in $Q_1$;
 \item\label{item: smooth outside}  {\bf [Regularity outside blow up locus] } $\rho\in C^{1+\alpha,(1+\alpha)/2}(Q_1\setminus \Sigma)$.
\end{enumerate}
\end{thm}
This covers most contents of Theorem \ref{main result}, except that $m_j(t)=8\pi$ (the $8\pi$ phenomena) and the dynamical law \eqref{dynamical law}. These two properties will be established below in Proposition \ref{prop multiplicity one} and Theorem \ref{thm limit eqn} respectively.
\begin{proof}
  (A)
Choose a countable, dense subset $\mathcal{D}\subset(-1,1)$. By the mass bound \eqref{local uniform mass bound} and a diagonal argument, we can   choose a subsequence of
$i$ so that for each $t\in\mathcal{D}$, $u_i(t)dx$  converges weakly as Radon measures
to a limit Radon measure $\mu_t$. In the following,
such a subsequence of $i$ will be fixed, and by abusing notations, it will still be written as $i$. We will prove that the convergence in \eqref{3.2} holds for this subsequence.

By Lemma \ref{lem symmetrization}, the linear functionals
\begin{equation}\label{linear functional}
	\mathcal{L}_{i,t}(\psi):=\int_{B_1}u_i(x,t)\psi(x)dx,  \quad  \psi\in C_0^2(B_1),
\end{equation}
are   uniformly Lipschitz continuous on $(-1,1)$.  After passing to a subsequence, they converge to a limit function $\mathcal{L}_t(\psi)$
 uniformly on   $(-1,1)$, which is still  Lipschitz continuous
 on $(-1,1)$. Of course, if $t\in\mathcal{D}$, because $u_i(x,t)dx\rightharpoonup \mu_t$, we must have
\[\mathcal{L}_{t}(\psi)=\int_{B_1}\psi(x)d\mu_t.\]
Therefore, because $\mathcal{D}$ is dense in $(-1,1)$, the above convergence of $\mathcal{L}_{i,t}$
 holds for the entire sequence $i$ and we do not need to pass to a further subsequence of $i$.

For any $t\in(-1,1)$ (not necessarily in $\mathcal{D}$), for any further subsequence $i^\prime$ with $u_{i^\prime}(x,t)dx$ converging weakly to a Radon measure $\mu_t$, because $C_0^2(B_1)$ is dense in $C_0(B_1)$, we must have
\[\mu_t=\mathcal{L}_t.\]
Here we have used Riesz representation theorem to identify Radon measures with positive linear functionals on $C_0(B_1)$.
 As a consequence, $\mu_t$ is independent of the choice of subsequences, that is, for any $t\in(-1,1)$ and the entire sequence $i$,
\[ u_i(x,t)dx\rightharpoonup \mu_t.\]

(B) By the Lipschitz continuity of $\mathcal{L}_t(\psi)$ (when $\psi\in C_0^2(B_1)$), $\mu_t$ is continuous in $t$ with respect to the weak topology.

(C) For any $t\in(-1,1)$, define
\[\Sigma_t:=\{x: ~~ \mu_t(\{x\})\geq \varepsilon_\ast/2\}.\]
Hence there is an atom of $\mu_t$ at each point in $\Sigma_t$, whose mass is at least $\varepsilon_\ast/2$.
Because $\mu_t(B_1)\leq M$, $\Sigma_t$ is a finite set.

If $x\notin \Sigma_t$, then there exists an  $r<1-|x|$ such that $\mu_t(B_r(x))<\varepsilon_\ast/2$. By the convergence of $u_i(x,t)dx$, for all $i$ large,
\[ \int_{B_r(x)}u_i(y,t)dy<\varepsilon_\ast.\]
By Theorem \ref{thm ep regularity},  $u_i$ are uniformly bounded in $C^{1+\alpha,(1+\alpha)/2}(Q_{\theta_\ast r}(x,t))$. Hence they converge in $C(Q_{\theta_\ast r}(x,t))$. As a consequence, $\mu_t=\rho(x,t)dx$ in $B_{\theta_\ast r}(x)$ for some   function $\rho\in C^{1+\alpha,(1+\alpha)/2}(Q_{\theta_\ast r}(x,t))$. In conclusion, $\mu_t$ is absolutely continuous with respect to Lebesgue measure outside $\Sigma_t$.

(D)  Since each $\Sigma_t$ is a finite set, we need only to show that\\
{\bf Claim.}   If $t_i\to t$, $x_i\in \Sigma_{t_i}$ and $x_i\to x$,
  then $x\in\Sigma_t$.

Indeed, assume by the contrary that $x\notin \Sigma_t$, then there exists an $r>0$ such that
  \[ \mu_t(B_r(x))<\varepsilon_\ast/2.\]
  Take a cut-off function $\psi$ subject to
  $B_{r/2}(x)\subset B_r(x)$. By the Lipschitz continuity of
  $\mathcal{L}_t(\psi)$, there exists a small $\delta>0$ such that for any $s\in(t-\delta,t+\delta)$,
  \[\mu_s(B_{r/2}(x))< \varepsilon_\ast/2.\]
  This implies that for all of these $s$, $B_{r/2}(x)$ is disjoint from $\Sigma_s$. This is a contradiction with the fact that $x_i\in\Sigma_{t_i}$ and $(x_i,t_i)\to (x,t)$. The claim follows.

(E) The  regularity of $\rho$ in $Q_1\setminus \Sigma$ has already been established in the proof of (C).
\end{proof}

Because
\[\nabla v_i(x,t)=-\frac{1}{2\pi}\int_{B_1}\frac{x-y}{|x-y|^2}u_i(y,t)dy,\]
by the convergence of $u_i(x,t)dx$ established in this theorem, we see
$\nabla v_i$ converges to
\begin{equation}\label{representation of v limit}
\nabla v_\infty(x,t):=-\sum_{j=1}^{N(t)}\frac{m_j(t)}{2\pi}\frac{x-q_j(t)}{|x-q_j(t)|^2}-\frac{1}{2\pi}\int_{B_1}\frac{x-y}{|x-y|^2}\rho(y,t)dy.
\end{equation}
By standard estimates on Newtonian potential,  this convergence is strong in $L^q(B_1)$ for any $q<2$, and uniform in any compact set of $B_1\setminus \Sigma_t$.

\section{The $8\pi$ phenomena}\label{sec multiplicity one}
\setcounter{equation}{0}

This section is devoted to the proof of the following result, which says that the mass of each atom in $\mu_t$ is exactly $8\pi$. We keep the notations used in Theorem \ref{thm convergence}.
\begin{prop}\label{prop multiplicity one}
  For any $t\in(-1,1)$ and each $q_j(t)\in\Sigma_t$, $m_j(t)=8\pi$.
\end{prop}
The proof of this proposition relies on the blow up method. By blowing up $(\mu_t)$, we can reduce the problem into a simpler setting where the limiting measure is \emph{a single} Dirac measure. This then makes the second momentum calculation applicable. 

\begin{defi}\label{definition of blow up for limiting measures}[Blow-up of $\mu$ and $\Sigma$]
For any $(q_j(t),t)\in\Sigma$ and $r>0$, let
\[ \widehat{\mu}_s^r(A):= \mu_{t+r^2s}(q_j(t)+rA),  \quad \forall A\subset \R^2 \quad \text{Borel},\]
and
\[\widehat{\Sigma}^r:= \left\{(y,s):  (q_j(t)+ry, t+r^2s)\in \Sigma\right\}.\]
\end{defi}

\begin{proof}[Proof of Proposition \ref{prop multiplicity one}] Take an arbitrary $t\in(-1,1)$ and a point $q_j(t)\in \Sigma_t$.  Denote 
	\[m:=m_j(t).\]
The remaining proof  is divided into three steps.
	
 {\bf Step 1. Reduction to a simple setting.}
In accordance with the definition of $\widehat{\mu}^r_s$, define
\[
\left\{
\begin{aligned}
	& u_i^r(y,s):= r^2 u_i\left(q_j(t)+ry, t+r^2s\right),\\
	& \nabla v_i^r(y,s):= r\nabla v_i\left(q_j(t)+ry, t+r^2s\right),\\
	& \nabla f_i^r(y,s):= r\nabla f_i\left(q_j(t)+ry, t+r^2s\right),\\
	&g_i^r(y,s):=r^2 g_i\left(q_j(t)+ry, t+r^2s\right).
\end{aligned}\right.
\]
Then $u_i^r(y,s)dy$ converges to $\widehat{\mu}_s^r$  in
 the sense of Theorem \ref{thm convergence} Item \ref{item:limit}.

As in the proof of Theorem  \ref{thm convergence} Item \ref{item:limit}, we can take a subsequence $r_i\to0$ so that  for any $s\in\R$, $\widehat{\mu}_s^{r_i}$ converges weakly to $\widehat{\mu}_s$ as finite Radon measures on $\R^2$. Then by a diagonal argument, we can take a further subsequence of $i$ (not relabelling) so that for any $s\in\R$, $u_i^{r_i}(y,s)dy$ converges to $\widehat{\mu}_s$ as finite Radon measures on $\R^2$. In the following we denote $\widehat{u}_i:=u_i^{r_i}$.

  Since $\Sigma_t$ is a finite set, there exists an $r>0$ such that
$\Sigma_t\cap B_r(q_j(t))=\{q_j(t)\}$. Then by the structure of $\mu_t$ in Theorem \ref{thm convergence} Item \ref{item:support}, we get
\[\widehat{\mu}^r_0=m\delta_0+ r^2\rho(rx,0)dx.\]
Because $\rho\in L^1(B_1)$, sending $r\to0$ we obtain
\begin{equation}\label{4.1}
   \widehat{\mu}_0=m\delta_0.
\end{equation}

{\bf Step 2. Clearing out.}  In this step we prove the following claim, which can be understood as a clearing out property for the blow up behavior of $\widehat{u}_i$.

{\bf Claim 1.} $\widehat{u}_i$ does not blow up in the annular cylinder $\left(B_{3/4}\setminus B_{3/8}\right)\times (-t_\ast,t_\ast)$ for some $t_\ast>0$.

Take a cut-off function $\psi_1$ subject to $\left(B_{7/8}\setminus B_{1/4}\right)\subset \left(B_1\setminus B_{1/8}\right)$.
 By Lemma \ref{lem symmetrization} and our assumption on the uniform bound on $\nabla f_i$ and $g_i$,
\[ \left|\frac{d}{dt}\int_{B_1}\widehat{u}_i(x,t)\psi_1(x)dx\right|\leq C\left(M+M^2\right)+o_i(1).\]
Combining this inequality with \eqref{4.1}, we deduce that for any $t$ satisfying
\[ |t|\leq t_\ast:=\frac{\varepsilon_\ast}{4C\left(M+M^2\right)},\]
it holds that
\[ \int_{B_{7/8}\setminus B_{1/4}}\widehat{u}_i(x,t)dx\leq \frac{\varepsilon_\ast}{3}.\]
By Theorem \ref{thm ep regularity} and Arzela-Ascoli theorem,  $\widehat{u}_i$ converges uniformly to
a function $\widehat{\rho}$ in $\left(B_{3/4}\setminus B_{3/8}\right)\times (-t_\ast,t_\ast)$, and the claim follows.

{\bf Step 3. Second momentum calculation.} Take a cut-off function $\psi_2$ subject to $B_{1/2}\subset B_{5/8}$. Consider the second momentum
\[ M(s):=\int_{B_1}|x|^2\psi_2(x)d\widehat{\mu}_s.\]

{\bf Claim 2.} $M(s)$ is differentiable in $(-t_\ast,t_\ast)$, and
\begin{equation}\label{4.2}
 M^\prime(0)=4m-\frac{1}{2\pi}m^2.
\end{equation}
By definition $M(s)\geq 0$, while by \eqref{4.1}, $M(0)=0$. The differentiability in this claim then implies that
 \[ M^\prime(0)=0.\]
This identity combined with \eqref{4.2} implies that $m=8\pi$. Hence the proof of this proposition is complete, provided this claim is true.

{\bf Proof of Claim 2.} Recall the definition of $\Theta_\psi$ for the symmetrization of a $C^2$ function $\psi$ in \eqref{symmetrization of test fct}. Then a direct calculation shows that the symmetryzation of  $|x|^2\psi_2(x)$ is given by
\[ \Theta_{|x|^2\psi_2(x)}(x,y) =2\psi_2(x)\psi_2(y)+\Phi(x,y),\]
where
\begin{eqnarray}\label{def of Phi}
 \Phi(x,y)&:=&
2\frac{x-y}{|x-y|^2}\cdot
  \left[ x\psi_2(x)\left(1-\psi_2(y)\right)-y\psi_2(y)\left(1-\psi_2(x)\right)\right] \nonumber\\
  &+&\frac{x-y}{|x-y|^2}\cdot
  \left[|x|^2\nabla\psi_2(x)-|y|^2\nabla\psi_2(y)\right].
\end{eqnarray}
Let
 \[ M_i(s):=\int_{B_1}|x|^2\psi_2(x)\widehat{u}_i(x,s)dx.\]
  By \eqref{symmetrization identity}, we obtain
\begin{eqnarray*}
M_i^\prime(s) &=& 4\int_{B_1}\psi_2(x)\widehat{u}_i(x,s)dx
 -\frac{1}{2\pi}\left[\int_{B_1}\psi_2(x)\widehat{u}_i(x,s)dx\right]^2\\
  &+&\int_{B_1}\left[|x|^2\Delta\psi_2(x)+4x\cdot\nabla\psi_2(x)\right]\widehat{u}_i(x,s)dx\\
  &+& \int_{B_1} \nabla \widehat{f}_i(x,s)\cdot\nabla\left(|x|^2\psi_2(x)\right)\widehat{u}_i(x,s)dx+\int_{B_1}\widehat{g}_i(x,t)|x|^2\psi_2(x)dx\\
  &-&\frac{1}{4\pi}\int_{\R^2}\int_{\R^2}\Phi(x,y)\widehat{u}_i(x,s)\widehat{u}_i(y,s)dxdy\\
  &=:&\mathrm{I}_i+\mathrm{II}_i+\mathrm{III}_i+\mathrm{IV}_i+\mathrm{V}_i+\mathrm{VI}_i.
  \end{eqnarray*}
 Let us analyse the convergence of these six terms one by one.
\begin{enumerate}
  \item By the weak convergence of $\widehat{u}_i(x,s)dx$, $\mathrm{I}_i$ converges uniformly in $(-t_\ast,t_\ast)$
  to
  \[4\int_{B_1}\psi_2(x)d\widehat{\mu}_s(x).\]

  \item  In the same way, $\mathrm{II}_i$ converges uniformly in $(-t_\ast,t_\ast)$
  to
  \[-\frac{1}{2\pi}\left[\int_{B_1}\psi_2(x)d\widehat{\mu}_s(x)\right]^2.\]

  \item Similarly, $\mathrm{III}_i$
  converges uniformly  in $(-t_\ast,t_\ast)$ to
  \[\int_{B_1}\left[|x|^2\Delta\psi_2(x)+4x\cdot\nabla\psi_2(x)\right]d\widehat{\mu}_s.\]

  \item Because $\|\nabla\widehat{f}_i\|_{L^\infty}\leq Cr_i$, $\mathrm{IV}_i$ converges uniformly to $0$.

\item Because $\|\widehat{g}_i\|_{L^\infty(Q_1)}\leq Cr_i^2$, $\mathrm{V}_i$ converges uniformly to $0$.
  \item
 Concerning $\mathrm{VI}_i$,  first because we have established the uniform convergence of $\widehat{u}_i$ in $\left(B_{3/4}\setminus B_{3/8}\right)\times (-t_\ast,t_\ast)$ in Step 2, an application of the dominated convergence theorem gives us the uniform in $s\in(-t_\ast,t_\ast)$ convergence
 \begin{eqnarray*}
 	&&\int_{B_{3/4}\setminus B_{3/8}}\int_{B_{3/4}\setminus B_{3/8}}\Phi(x,y)\widehat{u}_i(x,s)\widehat{u}_i(y,s)dxdy\\
&&  \to \int_{B_{3/4}\setminus B_{3/8}}\int_{B_{3/4}\setminus B_{3/8}}\Phi(x,y)d\widehat{\mu}_s(x)d\widehat{\mu}_s(y).
\end{eqnarray*}
 Next, by the form of $\Phi$ in \eqref{def of Phi} and using the facts that $\psi_2\equiv 1$ in $B_{1/2}$, $\psi_2\equiv 0$ outside $B_{5/8}$, we see $\Phi$ is continuous outside $\left(B_{3/4}\setminus B_{3/8}\right)\times \left(B_{3/4}\setminus B_{3/8}\right)$.\footnote{In fact, the only trouble appears on the diagonal $\{x=y\}$, but  it can be directly checked that $\Phi(x,y)=0$ if $|x-y|\leq 1/16$. So there is no effect from the diagonal.} Then by the weak convergence of $\widehat{u}_i$, the remaining part in $\mathrm{VI}_i$ converges uniformly in
 $(-t_\ast,t_\ast)$ to
 \[ \int\int_{\left(\R^2\times\R^2\right)\setminus\left[ \left(B_{3/4}\setminus B_{3/8}\right)\times\left(B_{3/4}\setminus B_{3/8}\right)\right]}\Phi(x,y)d\widehat{\mu}_s(x)d\widehat{\mu}_s(y).\]
In conclusion,  $\mathrm{VI}_i$ converges uniformly to
  \[ \int_{\R^2}\int_{\R^2}\Phi(x,y)d\widehat{\mu}_s(x)d\widehat{\mu}_s(y).\]
\end{enumerate}

Putting these six facts together, we deduce that $M_i(s)$ converges  to $M(s)$ in $C^1(-t_\ast,t_\ast)$. Moreover,
\begin{eqnarray}\label{derivative of second momentum}
  M^\prime(s)&=&4\int_{B_1}\psi_2(x)d\widehat{\mu}_s(x)-\frac{1}{2\pi}\left[\int_{B_1}\psi_2(x)d\widehat{\mu}_s(x)\right]^2 \nonumber\\
  &+&\int_{B_1}\left[|x|^2\Delta\psi_2(x)+4x\cdot\nabla\psi_2(x)\right]d\widehat{\mu}_s\\
  &-&\frac{1}{4\pi}\int_{\R^2}\int_{\R^2}\Phi(x,y)d\widehat{\mu}_s(x)d\widehat{\mu}_s(y). \nonumber
\end{eqnarray}
Evaluating this identity at $s=0$, and then using \eqref{4.1} and the definition of $\psi_2$, we get \eqref{4.2}. The proof of the claim is thus complete.
 \end{proof}

 \begin{rmk}\label{rmk 4.2}
 The Claim 1 established in the above proof is a clearing out property. In the below we will often use it in the  	following form: for any $t$ and $(x,t)\in\Sigma_t$, there exists an $r>0$ such that
 \[\Sigma\cap Q_r(x,t)\subset B_{r/2}(x)\times(t-r^2, t+r^2).\]
 In other words, all blow up points in $Q_r(x,t)$  lie near the line $\{x\}\times\R$.
 \end{rmk}

\section{Local structure of   blow up locus}\label{sec local structure}
\setcounter{equation}{0}

In this section we study the local structure of the blow up locus $\Sigma$.   Notice that  we are only interested in the local property, i.e. the blow up locus $\Sigma$ in a small parabolic cylinder $Q_r(x,t)$. Then by Remark \ref{rmk 4.2} and scaling the cylinder  $Q_r(x,t)$ to unit size, we may assume  $\Sigma\subset B_{1/2}\times(-1,1)$.

The main result of this section is
\begin{prop}\label{prop local structure}
For any $r_0<1$,
there exist finitely many functions $\xi_i: [-r_0^2,r_0^2]\mapsto B_{1/2}$ which are $1/2$-H\"{o}lder continuous, such that $\Sigma\cap Q_{r_0}\subset \cup_{i}\{(\xi_i(t),t)\}$.	
\end{prop}

 \subsection{Two technical lemmas}
 First  we need two technical lemmas, which can be viewed as  global version of Proposition \ref{prop local structure}. Below they will be used to study blow-up limits of $\mu_t$. 
 
 The first lemma is about a backward problem.
\begin{lem}\label{lem structure of backward weak sol}
Suppose a family of Radon measures on $\R^2$,  $\mu_t$, $t\in(-\infty,0]$ is a weak limit given  in Theorem \ref{thm convergence}, from a sequence of solutions $u_i$, where in the equation of $u_i$, the two terms  $\nabla f_i$ and $g_i$ converge to $0$ uniformly.
 Assume there exists a constant $M$ such that
\[\mu_t(\R^2)\leq M, \quad \forall  t\in(-\infty,0].\]
 If
$\mu_0=8\pi\delta_{0}$, then for any $t\leq 0$, $\mu_t=8\pi\delta_{0}$.
\end{lem}
\begin{proof}
{\bf Step 1.}  Let $\rho$ be the regular part in the Lebesgue decomposition of $\mu_t$  in Theorem \ref{thm convergence} Item \ref{item:support}.
We claim that $\rho\equiv 0$.

In fact, in the open set
$\Omega:=\left(\R^2\times(-\infty,0)\right)\setminus \Sigma$, $\rho$ is  a classical solution of a linear parabolic equation
(i.e. the first equation in \eqref{eqn}). By Theorem \ref{thm convergence} Item \ref{item:support}, $\Omega$ is connected.
By our assumption, $\rho(x,0)=0$ in $\R^2\setminus\{0\}$.
Hence we can apply the strong maximum principle to deduce  that $\rho\equiv 0$ in $\Omega$. Because for each $t$, $\rho(\cdot,t)\in L^1(\R^2)$ and $\Sigma_t$ is a finite set, we get the claim.

The detail on the application of the strong maximum principle is as follows. By Theorem \ref{thm convergence} Item \ref{item:support}, $\Omega$ is  path connected.
Then for any $(x,t)\in \Omega$, there exists a continuous curve $\gamma\subset\Omega$  connecting $(x,t)$ and $(y,0)$ for some $y\in\R^n\setminus\{0\}$. Moreover,  because each time slice of $\Sigma$ is a finite set,     the curve $\gamma$ can be chosen to be a graph over the $t$-axis. Then  along this curve we choose a chain of backward parabolic cylinders, $Q_{r_i}^-(x_i,t_i)$, $i=0$, $\dots$, $N$ for some $N\in\mathbb{N}$, so that 
\begin{enumerate}
	\item for each $i$, $Q_{r_i}^-(x_i,t_i)\subset\Omega$;
	\item $(x_0,t_0)=(y,0)$ and $(x,t)\in B_{r_N}(x_N)\times\{t_N-r_N^2\}$ (that is, $x\in B_{r_N}(x_N)$ and $t=t_N-r_N^2$);
	\item for each $i=1$, $\dots$, $N$, $B_{r_i}(x_i)\times \{t_i\}$ has nonempty intersection with $B_{r_{i-1}}(x_{i-1})\times\{t_{i-1}-r_i^2\}$ (in particular, $t_i=t_{i-1}-r_{i-1}^2$).
\end{enumerate}
Starting from the fact that $\rho(y,0)=0$, and then applying the strong maximum principle in $Q_{r_i}^-(x_i,t_i)$ inductively, we deduce that $\rho\equiv 0$ in every $Q_{r_i}^-(x_i,t_i)$. In the last step $i=N$, we get $\rho(x,t)=0$.
	
{\bf Step 2.}	
	For each $R>0$, take a cut-off function $\psi_R$ subject to $B_R\subset B_{2R}$. By Lemma \ref{lem symmetrization} (for $u_i$) and then passing to the limit using Theorem \ref{thm convergence} Item \ref{item:limit}, we get
	\[\left|\frac{d}{dt}\int_{\R^2}\psi_R(x)d\mu_t(x)\right|\leq \frac{C(M+M^2)}{R^2}.\]
Integrating this inequality in $t$, we obtain
	\[\left|\int_{\R^2}\psi_R(x)d\mu_t(x)-8\pi\right|\leq \frac{C(M+M^2)}{R^2}|t|.\]
Letting $R\to+\infty$, we deduce that
\begin{equation}\label{5.1}
	\mu_t(\R^2)=8\pi, \quad \forall  t\in(-\infty,0].
\end{equation}	
Combining this fact with Proposition \ref{prop multiplicity one} and the result in Step 1, we deduce that for each $t$, there exists exactly one point  $q(t)\in\R^2$ such that $\mu_t=8\pi\delta_{q(t)}$.
	
{\bf Step 3.}
	Take the cut-off function $\psi_R$ as in Step 2 and define
\[M_R(t):=\int_{\R^2}|x|^2\psi_R(x)d\mu_t(x)=8\pi |q(t)|^2 \psi_R(q(t)).\]
By the same derivation of \eqref{derivative of second momentum} as in the proof of Proposition \ref{prop multiplicity one}, we get
\begin{eqnarray*}
	M_R^\prime(t)&=&4\int_{\R^2}\psi_R(x)d\mu_t(x)-\frac{1}{2\pi}\left[\int_{\R^2}\psi_R(x)d\mu_t(x)\right]^2 \\
	&+&\int_{\R^2}\left[|x|^2\Delta\psi_R(x)+4x\cdot\nabla\psi_R(x)\right]d\mu_t(x)\\
	&-&\frac{1}{4\pi}\int_{\R^2}\int_{\R^2}\Phi(x,y)d\mu_t(x)d\mu_t(y),
\end{eqnarray*}
where $\Phi(x,y)$ is defined  in \eqref{def of Phi}.

For any $T>0$, if $R$ has been chosen  large enough so that for any $t\in[-T,0]$, $q(t)\in B_R$, by using the fact that $\mu_t=8\pi\delta_{q(t)}$ and the form of $\Phi$ etc., we deduce that
\[M_R^\prime(t)\equiv 0, \quad \forall t\in[-T,0].\]
By the assumption that $\mu_0=8\pi\delta_0$, we also have $M_R(0)=0$. Hence $M_R(t)\equiv 0$. This then implies that $q(t)\equiv 0$.
\end{proof}
The second lemma is about a forward problem.  This result should be compared with the ones obtained in Wei \cite{WeiDongyi} and Fournier-Tardy \cite{Fournier-Tardy2}, but it should be noticed that Wei considered only $L^1$ initial data in the setting of Cauchy  problem for \eqref{eqn}, while Fournier-Tardy used a little different notion of measure valued weak solutions.
\begin{lem}\label{lem structure of forward weak sol}
Suppose a family of Radon measures on $\R^2$,  $\mu_t$, $t\in[0,+\infty)$ is  a weak limit given  in Theorem \ref{thm convergence}, from a sequence of solutions $u_i$,  where in the equation of $u_i$, the two terms  $\nabla f_i$ and $g_i$ converge to $0$ uniformly.
 Assume there exists a constant $M$ such that
\[\mu_t(\R^2)\leq M, \quad \forall  t\in[0,+\infty).\]
 If
	$\mu_0=8\pi\delta_{0}$, then for any $t\geq 0$, $\mu_t=8\pi\delta_{0}$.
\end{lem}
\begin{proof}
First as in Step 2 in the proof of the previous lemma, we still have \eqref{5.1} (for any $t\in[0,+\infty)$).	Combining this fact with  Proposition \ref{prop multiplicity one}, we see for any $t>0$,
\begin{itemize}
	\item {\bf Alternative I:} either there exists a $q(t)\in\R^2$ such that $\mu_t=8\pi \delta_{q(t)}$ ;
	\item {\bf Alternative II:} or there exists a   positive smooth function $\rho(t)$ on $\R^2$ such that  $\mu_t=\rho(x,t)dx$.
\end{itemize}
Define
\[T_\ast:=\sup\{t: ~ \mbox{Alternative I holds at }  t\}.\]
Similar to Step 1 in the proof of the previous lemma, we deduce that  if $T_\ast>0$, then for any $t\in[0,T_\ast]$, Alternative I holds.

We claim   that
\begin{equation}\label{5.2}
 T_\ast=+\infty.
\end{equation}
{\bf Proof of \eqref{5.2}.}
	Assume by the contrary that $T_\ast<+\infty$. Then in $\R^2\times(T_\ast,+\infty)$, $\mu_t=\rho(x,t)dx$, where $\rho$ is a smooth solution of \eqref{eqn}.
	For any $t>T_\ast$ and $R>1$,  take a cut-off function $\psi_R$ subject to $B_R\subset B_{2R}$ and define
	\[M_R(t):=\int_{\R^2}|x|^2\psi_R(x)\rho(x,t)dx, \quad \mathcal{E}_R(t):=\int_{B_R^c}\rho(x,t)dx.\]
	By \eqref{symmetrization of test fct}, we have
	\begin{eqnarray}\label{MR derivative}
		M_R^\prime(t)&=& 4\int_{\R^2}\psi_R(x)\rho(x,t)dx -\frac{1}{2\pi} \left[\int_{\R^2}\psi_R(x)\rho(x,t)dx\right]^2 \nonumber\\
		&+& \int_{\R^2}\left[4x\cdot\nabla\psi_R(x)+|x|^2\Delta\psi_R(x)\right]\rho(x,t)dx \nonumber \\
		&+&\frac{1}{2\pi} \int_{\R^2} \int_{\R^2}\frac{(x-y)\cdot x}{|x-y|^2}\psi_R(x)\left[1-\psi_R(y)\right] \rho(x,t)\rho(y,t)dxdy\\
		&-&\frac{1}{2\pi} \int_{\R^2} \int_{\R^2}\frac{(x-y)\cdot y}{|x-y|^2}\psi_R(y)\left[1-\psi_R(x)\right] \rho(x,t)\rho(y,t)dxdy \nonumber\\
		&-&\frac{1}{4\pi} \int_{\R^2} \int_{\R^2}\frac{x-y}{|x-y|^2}\cdot \left[|x|^2\nabla\psi_R(x)-|y|^2\nabla \psi_R(y)\right] \rho(x,t)\rho(y,t)dxdy. \nonumber
	\end{eqnarray}
	
By \eqref{5.1},
\[\int_{\R^2}\rho(x,t)dx=8\pi.\]
Thus
\[\int_{\R^2}\psi_R(x)\rho(x,t)dx =8\pi-\int_{\R^2}\left[1-\psi_R(x)\right]\rho(x,t)dx.\]
	Plugging this identity into \eqref{MR derivative} and noticing that $1-\psi_R(x)\equiv 0$ in $B_R$, we obtain
	\[|M_R^\prime(t)|\leq C\mathcal{E}_R(t).\]
	Because $M_R(T_\ast)=0$, this inequality implies that for any fixed $t>T_\ast$,
	\[ M_R(t)\leq C(t-T_\ast) \max_{T_\ast\leq s \leq t}\mathcal{E}_R(s).\]
	By the definition of  $\psi_R$, $M_R(t)$ is non-decreasing in $R$. On the other hand, we have
	\[\lim_{R\to+\infty}\max_{T_\ast\leq s \leq t}\mathcal{E}_R(s)=0.\]
Thus
	\begin{eqnarray*}
	M_R(t) &\leq & \lim_{R\to+\infty}  M_R(t) \\
		&\leq & C(t-T_\ast) \lim_{R\to+\infty}  \max_{T_\ast\leq s \leq t}\mathcal{E}_R(s)  \\
		&=& 0.
	\end{eqnarray*}
	This is possible only if $\mu_t$ is a Dirac measure at $0$, which is a contradiction with our assumption. In other words, \eqref{5.2} must hold.
	
	Finally, once we have shown that $\mu_t=8\pi\delta_{q(t)}$ for any $t>0$, following the argument  in Step 3 in the proof of the previous lemma, we deduce that for any $t>0$, $M_R(t)=0$ for all $R$ large enough. This implies that $q(t)\equiv 0$.
\end{proof}
\begin{rmk}
	The above proof can be used to show that, there does not exist forward self-similar solutions of the Keller-Segel system \eqref{eqn} with mass not smaller than $8\pi$. More precisely, there does not exist solutions satisfying the following three conditions:
	\begin{enumerate}
		\item $u\in C^{2,1}(\R^n\times(0,+\infty))$ is a solution of \eqref{eqn};
		\item it is
		forward self-similar  in the sense that
		\[u(\lambda x,\lambda^2t)=\lambda^2 u(x,t), \quad \forall \lambda>0;\]
		\item the total mass satisfies
		\[\int_{\R^2}u(x,t)dx\equiv m\geq 8\pi.\]
	\end{enumerate}
	Indeed,  under these assumptions, the initial condition is $u(0)=m\delta_0$. The above calculation in the proof of \eqref{5.2} then leads to a contradiction. On the other hand, if $m<8\pi$, such  forward self-similar solutions do exist, see Naito-Suzuki \cite{Naito-Suzuki2004}.
\end{rmk}

\subsection{Proof of Proposition \ref{prop local structure}}

To prove Proposition \ref{prop local structure},  we need the following characterization of the limit of the blow-up sequences $\widehat{\mu}^r_s$ (see Definition \ref{definition of blow up for limiting measures}) as $r\to0$.
\begin{lem}\label{lem tangent measure}
For any $s\in\R$, $\widehat{\mu}^r_s\rightharpoonup 8\pi\delta_0$ as $r\to0$,
\end{lem}
\begin{proof}
 First, as in the proof of Theorem \ref{thm convergence} Item \ref{item:limit}, we can assume (after passing to a subsequence of $r\to0$) that, for any $s\in\R$, $\widehat{\mu}^r_s$ converges weakly to $\widehat{\mu}^0_s$ as Radon measures. Moreover, $\widehat{\mu}^0_s$, as a family of Radon measures, is a weak limit of a sequence of solutions $\widehat{u}_i$, where in the equation of $\widehat{u}_i$, the two terms $\nabla\widehat{f}_i$ and $\widehat{g}_i$ converge to $0$ uniformly.

By Theorem \ref{thm convergence} Item \ref{item:support} and Proposition \ref{prop multiplicity one}, for any $R>0$,
\begin{eqnarray*}
	\widehat{\mu}_0^r(B_R)&=&\mu_0\left(B_{R r}(q_j(t))\right)\\
	&=& 8\pi+\int_{B_{Rr}(q_j(t))}\rho(x,t)dx \to 8\pi \quad \text{as} ~ r\to0.
\end{eqnarray*}
Thus $\widehat{\mu}^0_0=8\pi\delta_0$.

Then by Lemma \ref{lem structure of backward weak sol}, for any $t\leq 0$, $\widehat{\mu}_t^0=8\pi\delta_{0}$, and by Lemma \ref{lem structure of forward weak sol}, for any $t\geq 0$, $\widehat{\mu}_t^0=8\pi\delta_{0}$.
\end{proof}
\begin{coro}\label{coro 6.5}
For any $(q_j(t),t)\in\Sigma$ and $r$ sufficiently small,
$ \Sigma\cap Q_r(q_j(t),t)$ belongs to an $o(r)$ neighborhood of  $\{q_j(t)\}\times (t-r^2,t+r^2)$.
\end{coro}
\begin{proof}
By the previous lemma,   for any $\delta>0$, if $r$ is sufficiently small, then for any $s\in(-1,1)$,
\[\widehat{\mu}^r_s(B_1\setminus B_{\delta})<8\pi.\]
Combining this fact with Proposition \ref{prop multiplicity one}.
we deduce that  
\[\widehat{\Sigma}^r\cap Q_1\subset B_\delta \times(-1,1).\]
 Scaling back to the original scale, this is the desired claim.
\end{proof}

\begin{lem}\label{lem uniform distance separating}
For any $r<1$, there exists a $\delta>0$ such that for any $t\in[-r^2,r^2]$ and  $x_1\neq x_2\in \Sigma_t$, we have
\[|x_1-x_2|\geq \delta.\]
\end{lem}
\begin{proof}
Assume by the contrary that there exist $t_i\in [-r^2,r^2]$ and $x_{i1}\neq x_{i2}\in\Sigma_{t_i}$ such that
\begin{equation}\label{absurd assumption 3}
	|x_{i1}-x_{i2}|\to0.
\end{equation}
After passing to a subsequence, we may assume $t_i\to t_\infty$, both $x_{i1}$ and $x_{i2}$ converge to $x_\infty$.

Because $\Sigma$ is closed (see Theorem \ref{thm convergence} Item \ref{item:blow up locus}), $x_\infty\in\Sigma_{t_\infty}$.
By Proposition \ref{prop multiplicity one}, 
\[\mu_{t_i}\geq 8\pi\delta_{x_{i1}}+8\pi\delta_{x_{i2}}.\]
By the weak continuity of $\mu_{t}$  (see Theorem \ref{thm convergence} Item \ref{item:continuity}) and the convergence of $x_{i1}$ and $x_{i2}$, passing to the limit in the above inequality leads to
\[\mu_{t_\infty}\geq 16\pi\delta_{x_{\infty}}.\]
This is a contradiction with Proposition \ref{prop multiplicity one}.
\end{proof}
We also need a fact about the extension of H\"{o}lder continuous functions.
\begin{lem}\label{lem Holder extension}[Extension of H\"{o}lder continuous function]
	If $f$ is an $\alpha$-H\"{o}lder continuous function (for some $\alpha\in(0,1)$) on a closed subset $D\subset\R^n$, then it can be extended to  an $\alpha$-H\"{o}lder continuous function on the whole $\R^n$.
\end{lem}
\begin{proof}
	If $f$ is  a scalar function, the extension can be defined as
	\[\widetilde{f}(x):= \inf_{y\in D} \left[f(y)+L|x-y|^\alpha\right],\]
	where $L$ is the H\"{o}lder constant of $f$.
	
	If $f$ is a vector valued function, we can define its extension by considering each of its component as above.
\end{proof}

Now let us prove Proposition \ref{prop local structure}. 
\begin{proof}[Proof of Proposition \ref{prop local structure}]
The proof is divided into two steps.	In the following we fix an $r_0<1$ and restrict our attention to $B_1\times[-r_0^2,r_0^2]$. Recall that we always assume that $\Sigma\subset B_{1/2}\times(-1,1)$ (see Remark \ref{rmk 4.2}).
	
{\bf Step 1. H\"{o}lder curves in small intervals.} Take an arbitrary $t\in[-r_0^2,r_0^2]$ and $x_j\in \Sigma_t$. By Theorem \ref{thm convergence} Item \ref{item:support} and Proposition \ref{prop multiplicity one},
there exists an $r_j>0$ such that
\begin{equation}\label{one Dirac assumption}
	\left\{
\begin{aligned}
	&\Sigma_t\cap B_{r_j}(x_j)=\{x_j\},\\
	&8\pi\leq \mu_t(B_{r_j}(x_j))<8\pi+1/4.
\end{aligned}\right.
\end{equation}

	By the weak continuity of $\mu_t$, after shrinking $r_j$ further, we may assume  that for any $s\in(t-r_j^2,t+r_j^2)$,
	\[ \mu_s(B_{r_j/2}(x_j))<8\pi+1/2.\]
	By Proposition \ref{prop multiplicity one}, for these $s$, there exists at most one atom of $\mu_s$ in $B_{r_j/2}(x_j)$.
	
Let $\mathcal{I}_j\subset(t-r_j^2, t+r_j^2)$ be the set containing those  $s$ such that $\mu_s$ consists of an atom inside $B_{r/2}(x_j)$. 
Then there exists a function $q_j: ~\mathcal{I}_j \mapsto B_{r_j/2}(x_j)$ such that
\begin{equation}\label{6.02}
	\Sigma\cap Q_{r_j}(x_j,t)=\left\{(q_j(s),s): ~~ s\in \mathcal{I}_j\right\}.
\end{equation}
Because $\Sigma$ is a closed  set, $\mathcal{I}_j$ is relatively closed in $(t-r_j^2,t+r_j^2)$.

{\bf Claim 1. } $q_j$ is $1/2$-H\"{o}lder continuous on $\mathcal{I}_j$. \\
%By Corollary \ref{coro 6.5}, for any $s\in I$, there exists an open neighborhood of $s$, denoted by $J_s$, such that for any $s_1\neq s_2\in J_s\cap I$,
%\begin{equation}\label{6.01}
%|q_j(s_1)-q_j(s_2)|\leq |s_1-s_2|^{1/2}.
%\end{equation}
%These open intervals $J_s$, $s\in I$, form an open covering of the closed set $I$. Take a finite sub-cover $J_{s_i}$ from it.
We prove this claim by contradiction. Let
\[L_i:= \sup_{\substack{t_1,t_2\in \mathcal{I}_j\\ |t_1-t_2|>1/i}}	\frac{|q_j(t_{1})-q_j(t_{2})|}{|t_{1}-t_{2}|^{1/2}}.\]
Assume by the contrary that
\[L_i\to+\infty \quad \text{as} \quad i\to+\infty.\]
Because $|q_j(s)-x_j|<r_j/2$ for any $s$,  this is possible only if there exist two sequences $t_{i1}$, $t_{i2}\in \mathcal{I}_j$ such that $|t_{i1}-t_{i2}|\to0$ and
\begin{equation}\label{absurd assumption 2}
	\frac{|q_j(t_{i1})-q_j(t_{i2})|}{|t_{i1}-t_{i2}|^{1/2}}\to+\infty.
\end{equation}
As in Definition \ref{definition of blow up for limiting measures}, define the blow-up sequence
\[\widehat{\mu}^i_s(A):= \mu_{t_{i1}+|t_{i1}-t_{i2}| s}(q_j(t_{i1})+|t_{i1}-t_{i2}|^{1/2} A), \quad \forall  s\in\R ~\text{and}~A\subset\R^2 ~ \text{Borel}.\]
Then by \eqref{one Dirac assumption} and Lemma \ref{lem tangent measure},  $\widehat{\mu}^i_s\rightharpoonup 8\pi\delta_0$ for any $s\in\R$.
As in the proof of Corollary \ref{coro 6.5}, this implies that
\[	\frac{|q_j(t_{i1})-q_j(t_{i2})|}{|t_{i1}-t_{i2}|^{1/2}}\to0.\]
This is a contradiction with \eqref{absurd assumption 2}, and the claim follows.

%We extend the definition of $q_j$ to $(t-r_j^2, t+r_j^2)$ by Lemma \ref{lem Holder extension}. Because $\Sigma_t$ consists of $N(t)$ points,  the number of these curves is $N(t)$. Moreover, these curves do not intersect with each other (perhaps after shrinking $\delta$).

{\bf Step 2. Gluing local curves}. 
Because $\Sigma_t$ is a finite set, for any $t\in[-r_0^2,r_0^2]$,
\begin{equation}\label{6.01}
	r(t):=\min\left\{\min_{x_j\in\Sigma_t}r_j(t),\delta/100\right\}>0.
\end{equation}
Then 
\[\Sigma\cap (B_1\times(t-r(t)^2,t+r(t)^2))\subset \bigcup_{j=1}^{N(t)}\left\{(q_j(s),s): ~~ s\in \mathcal{I}_j\cap (t-r(t)^2,t+r(t)^2)\right\}.\]

The open intervals $(t-r(t)^2, t+r(t)^2)$ form a covering of $[-r_0^2,r_0^2]$. Take a finite sub-covering from it, denoted by $\{(t_\alpha^-, t_\alpha^+)\}$, $\alpha=1$, $\dots$, $K$, where
\[ t_\alpha^\pm=t_\alpha\pm r_\alpha^2 \quad \text{with} \quad t_\alpha=\frac{r_\alpha^++r_\alpha^-}{2}, ~~ r_\alpha=r(t_\alpha).\]
Without loss of generality, assume $t_1^-<t_2^-<\cdots<t_K^-$.

For each $\alpha$, denote $N_\alpha:=N(t_\alpha)$. For $x_{\alpha,j}\in \Sigma_{t_\alpha}$, $j=1$, $\dots$, $N_\alpha$, let $Q_{\alpha,j}:=Q_{r_\alpha}(x_{\alpha,j}, t_\alpha)$. Each of these cylinders satisfies the property \eqref{6.02}.

{\bf Claim 2.} For any $\alpha=1$, $\dots$, $K$ and $j=1$, $\dots$, $N_\alpha$, $Q_{\alpha,j}$ has nonempty intersection with at most one cylinders in the group $\alpha\pm 1$.

Assume for example that $Q_{\alpha,j}$ has nonempty intersection with $Q_{\alpha+1,k}$ and $Q_{\alpha+1,\ell}$ for two different $k$, $\ell$. Then by the triangle inequality,
\begin{align*}
	|x_{\alpha+1,k}-x_{\alpha+1,\ell}|&=\text{dist}_p((x_{\alpha+1,k},t_{\alpha+1}), (x_{\alpha+1,\ell},t_{\alpha+1}))\\
	&\leq \text{diam}_p(Q_{\alpha+1,k})+\text{diam}_p(Q_{\alpha,j})+\text{diam}_p(Q_{\alpha+1,\ell})<\delta,
\end{align*}
where $\text{diam}_p$ denotes the diameter in the paraboic distance, and we have used the facts that $\text{diam}_p(Q_{\alpha+1,k})<\delta/100$ by \eqref{6.01}.
This is a contradiction with  Lemma \ref{lem uniform distance separating}, and the claim follows.

By Claim 2, we can define two cylinders $Q_{\alpha,j}$ and $Q_{\beta,k}$ (without loss of generality, assume $\alpha<\beta$) to be connected, if there exist a chain of cylinders  $Q_{\alpha,j}$, $Q_{\alpha+1, j_1}$, $\dots$, $Q_{\beta,k}$, where two neighboring cylinders intersect with each other. 
Then the whole family $\{Q_{\alpha,j}\}_{\alpha,j}$ are divided into disjoint connected components. In each connected component, for any $\alpha$, there exists at most one cylinder $Q_{\alpha,j}$. Hence the unions of graphs $\{(q_{\alpha,j}(t),t)\}$ in each connected component is a graph  of a $1/2$-H\"{o}lder continuous function on a closed subset of $[-r_0^2,r_0^2]$. Notice that if $Q_{\alpha,j}\cap Q_{\alpha+1,k}\neq \emptyset$, by \eqref{6.02} (i.e. the uniqueness of atoms in $\mu_s$ inside these cylinders), we must have $q_{\alpha,j}=q_{\alpha+1,k}$ in the intersection of their definition domains.

Finally, we use Lemma \ref{lem Holder extension} to extend these functions to the whole $[-r_0^2,r_0^2]$. These are the desired $1/2$-H\"{o}lder continuous $\xi_i$.
\end{proof}

\begin{rmk}\label{rmk on N(t)}
	It is possible that $N(t)$ (the number of points in $\Sigma_t$) is not constant in time.
	If finite time blow up solutions can be continued past the singular time, then we get examples where $N(t)$  jumps up at the singular time. On the other hand, it seems highly impossible that a singularity can disappear once it has been created.\footnote{For example, Lemma \ref{lem structure of forward weak sol} says  that if  the equation is posed on the entire $\R^2$ and the initial value is $8\pi\delta_0$, then the solution will always be $8\pi\delta_0$, that is, it cannot be smoothed. } Hence it is natural to conjecture that $N(t)$ is non-decreasing and right continuous in $t$, but we do not know how to prove this.
\end{rmk}

\section{Limit equations}\label{sec limiting ODE}
\setcounter{equation}{0}

By the smooth convergence of $u_i$ in $Q_1\setminus\Sigma$, we see $\rho$ satisfies
\begin{equation}\label{diffustion part of limit}
	\rho_t-\Delta \rho=-\mbox{div}\left[\rho\left(\nabla v_\infty+\nabla f\right)\right] +g  \quad \mbox{in}~~ Q_1\setminus\Sigma.
\end{equation}
Here by Proposition \ref{prop multiplicity one},  the definition of $\nabla v_\infty$ (cf. \eqref{representation of v limit}) is
\begin{equation}\label{v in the limit 2}
	\nabla v_\infty(x,t)=-4\sum_{j=1}^{N(t)}\frac{x-q_j(t)}{|x-q_j(t)|^2}-\frac{1}{2\pi}\int_{B_1}\frac{x-y}{|x-y|^2}\rho(y,t)dy.
\end{equation}

In this section, we will establish the dynamical law for $\mu_t$ only in the case that $\rho$ is smooth enough and it satisfies \eqref{diffustion part of limit} in the entire $Q_1$.
The reason to add this assumption is that, only under this assumption, we can show that the blow up locus $\Sigma$ is regular in the following sense.
\begin{lem}\label{lem constancy}
	If $\rho\in C(Q_1)$, then there exists an $N\in\mathbb{N}$ such that $N(t)\equiv N$, that is,  for any $t\in(-1,1)$, $\Sigma_t$ consists of exactly $N$ distinct points, $q_1(t)$, $\dots$, $q_N(t)$. Moreover, each $q_j(t)$ is   continuous in $t$.
\end{lem}
\begin{proof}
	Because $\rho$ is continuous, $\rho(x,t)dx$, viewed as Radon measures, is continuous in the weak topology.  By Theorem \ref{thm convergence} Item \ref{item:continuity}, $\mu_t$ is also  continuous in the weak topology.
Hence by Theorem \ref{thm convergence} Item \ref{item:support} and Proposition \ref{prop multiplicity one},
	\[\mu_t-\rho(x,t)dx=\sum_{j=1}^{N(t)}8\pi \delta_{q_j(t)}\]
	is continuous in the weak topology. By testing this weak continuity with suitable compactly supported, continuous functions, we deduce that its total mass, which is $8\pi N(t)$, is constant in time. In other words,  there exists an $N\in\mathbb{N}$ such that  for any $t\in(-1,1)$, $\Sigma_t$ consists of exactly $N$ distinct points, $q_1(t)$, $\dots$, $q_N(t)$. The continuity of $q_j(t)$
also	 follows from  this weak continuity.
\end{proof}

Next we establish the dynamical law for these $q_j(t)$ under this assumption.
\begin{thm}\label{thm limit eqn}
	For each $j=1$, ..., $N$, $q_j$ is differentiable in $(-1,1)$. Moreover, it satisfies
	\begin{equation}\label{limit ODE 0}
		q_j^\prime(t)= 4\sum_{k\neq j}\dfrac{q_k(t)-q_j(t)}{|q_k(t)-q_j(t)|^2}+\frac{1}{2\pi}\int_{B_1}\frac{y-q_j(t)}{|y-q_j(t)|^2}\rho(y,t)dy+\nabla f(q_j(t),t).
	\end{equation}
\end{thm}
If $f=0$, this is \eqref{dynamical law} in Theorem \ref{main result}. The proof of Theorem \ref{main result}
is thus complete.

Before going into the detail of the proof, we explain briefly the idea. The main task is to single out the interaction between different atoms and the effect of the diffusion part $\rho$ to these atoms. In the following proof we achieve this goal by taking a localization  (defined in Lemma \ref{lem localization}) around a selected atom. We can also use Lemma \ref{lem symmetrization} directly, by calculating a local first momentum as in the proof of Proposition \ref{prop multiplicity one}. However, this will need a detailed knowledge of $\Theta_{x\eta}$ (with $\eta$ a suitable cut-off function). Nevertherless, by the localization argument, the self-interaction term, the interaction terms between different atoms as well as the interaction  between the selected atom  and the diffusion part can be revealed in a clear way, so we will use this approach in  the following proof.
\begin{proof}
	Take an arbitrary $t_0\in(-1,1)$ and a point $q_j(t_0)\in\Sigma_{t_0}$. By Lemma \ref{lem constancy},
	there exists an $r>0$ such that
	\begin{equation}\label{clearing out in annulus 2}
		\Sigma\cap Q_{2r}(q_j(t_0),t_0)=\{(q_j(t),t)\}.
	\end{equation}
	By Corollary \ref{coro 6.5}, after shrinking $r$ further, we may assume for any $t\in (t_0-r^2, t_0+r^2)$,  $q_j(t)\in B_{r/4}(q_j(t_0))$.
	
	Take a cut-off function $\eta$ subject to $B_{r/2}(q_j(t_0))\subset B_{r}(q_j(t_0))$. Following the localization procedure in Lemma \ref{lem localization}, set $\widetilde{u}_i:=u_i\eta$, and then define $\nabla\widetilde{v}_i$, $\nabla\widetilde{f}_i$, $\widetilde{g}_i$
	accordingly. By the convergence of $u_i$ etc., we have the following convergence results.
\begin{enumerate}
	\item For any $t\in(t_0-r^2, t_0+r^2)$, $\widetilde{u}_i(x,t)dx\rightharpoonup 8\pi\delta_{q_j(t)}+\widetilde{\rho}(x,t)dx$, where $\widetilde{\rho}(x,t):=\rho(x,t)\eta(x)$.
	\item By the above convergence of $\widetilde{u}_i$, we deduce that $\nabla\widetilde{v}_i$ converges to
	\[\nabla\widetilde{v}_\infty(x,t)-4\frac{x-q_j(t)}{|x-q_j(t)|^2}\]
in $L^q(B_1)$ for any $q<2$, and uniformly in any compact set of $Q_r(q_j(t_0),t_0)\setminus \{(q_j(t),t)\}$. Here
\[ \nabla\tilde{v}_\infty(x,t)=-\frac{1}{2\pi}\int_{B_1}\frac{x-y}{|x-y|^2}\widetilde{\rho}(y,t)dy.\]

	\item  By definition,
	\[\nabla v_i(x,t)-\nabla \widetilde{v}_i(x,t)=-\frac{1}{2\pi}\int_{B_1}\frac{x-y}{|x-y|^2}u_i(y,t)[1-\eta(y)]dy.\]
Notice that
\[u_i(y,t)[1-\eta(y)]dy \rightharpoonup 8\pi\sum_{k\neq j}\delta_{q_k(t)}+\rho(y,t)\left[1-\eta(y)\right]dy.\]
By \eqref{clearing out in annulus 2},  $q_k(t)\notin B_{2r}(q_j(t_0))$,
and  
$u_i$ converges uniformly to $\rho$ in $(B_{2r}(q_j(t_0))\setminus B_{r/4}(q_j(t_0)))\times(t_0-r^2,t_0+r^2)$. Then by standard estimates on Newtonian potential, $\nabla v_i(x,t)-\nabla \widetilde{v}_i(x,t)$ converges
	uniformly in $Q_r(q_j(t_0),t_0)$ to
\[-4\sum_{k\neq j}\frac{x-q_k(t)}{|x-q_k(t)|^2}-\frac{1}{2\pi}\int_{B_1}\frac{x-y}{|x-y|^2}\rho(y,t)[1-\eta(y)]dy.\]	
Combining this convergence with our assumption on  the convergence of $\nabla f_i$, we deduce that
	$\nabla\widetilde{f}_i$ converges uniformly  in $Q_r(q_j(t_0),t_0)$ to
\begin{equation}\label{definition of new f}
\nabla \widetilde{f}(x,t):=\nabla f(x,t)-4\sum_{k\neq j}\frac{x-q_k(t)}{|x-q_k(t)|^2}-\frac{1}{2\pi}\int_{B_1}\frac{x-y}{|x-y|^2}\rho(y,t)\left[1-\eta(y)\right]dy.
\end{equation}
	\item In the support of $\nabla\eta$, $u_i$, $\nabla u_i$ and $\nabla v_i$ converge uniformly to $\rho$, $\nabla\rho$ and $\nabla v_\infty$ respectively. Thus $\widetilde{g}_i(x,t)$ converges uniformly to
	\begin{eqnarray*}
		\widetilde{g}(x,t)&:=&g(x,t)\eta(x)+\rho(x,t) \left[\nabla v_\infty(x,t)+\nabla f(x,t)\right]\cdot \nabla\eta(x)\\
	&&-2\nabla \rho(x,t)\cdot\nabla\eta(x)-\rho(x,t)\Delta\eta(x).
\end{eqnarray*}
\end{enumerate}
	
By definition, for any $t\in(t_0-r^2,t_0+r^2)$, $\widetilde{u}_i(x,t)$ is compactly supported in $B_r(q_j(t_0))$.
	Plugging $\psi(x)=x$ as test function into \eqref{symmetrization identity}, by noting that $\Delta\psi\equiv 0$ and $\Theta_\psi\equiv 0$, we obtain
\[	 \frac{d}{dt}\int_{B_1}\psi(x)\widetilde{u}_i(x,t)dx=\int_{B_1}\left[\nabla \widetilde{f}_i(x,t)\cdot\nabla\psi(x) \widetilde{u}_i(x,t)+\widetilde{g}_i(x,t) \psi(x)\right]dx. \]
Sending $i\to+\infty$, by the  weak convergence of $\widetilde{u}_i$ and the uniform convergence of $\nabla\widetilde{f}_i$ and $\widetilde{g}_i$, we get
	\begin{eqnarray}\label{limit eqn 2}
		&&\frac{d}{dt}\left[\int_{B_1}\psi(x)\widetilde{\rho}(x,t)dx+8\pi q_j(t)\right]  \\
		&=&\int_{B_1}\nabla \widetilde{f}(x,t)\cdot\nabla \psi(x,t) \widetilde{\rho}(x,t)dx+\int_{B_1}\widetilde{g}(x,t)\psi(x)dx+8\pi\nabla\widetilde{f}(q_j(t),t). \nonumber
	\end{eqnarray}
We emphasize that at this stage, this differential identity holds only in the distributional sense.
	
	Because $\rho$ is a solution of \eqref{diffustion part of limit} in $Q_1$,  $\widetilde{\rho}=\rho\eta$ satisfies
	\[\widetilde{\rho}_t-\Delta\widetilde{\rho}=-\text{div}\left[\widetilde{\rho}	\left(\nabla\widetilde{v}_\infty-4\frac{x-q_j(t)}{|x-q_j(t)|^2}+\nabla\widetilde{f}\right)\right]+\widetilde{g}.\]
Then  by Lemma \ref{lem symmetrization}, we have
	\begin{eqnarray*}
		&&\frac{d}{dt}\int_{B_1}\psi(x)\widetilde{\rho}(x,t)dx \\
		&=& \int_{B_1}\left[-4\frac{x-q_j(t)}{|x-q_j(t)|^2}+\nabla \widetilde{f}(x,t)\right]\cdot\nabla \psi(x,t) \widetilde{\rho}(x,t)dx+\int_{B_1}\widetilde{g}(x,t)\psi(x)dx.
	\end{eqnarray*}
In view of the form of $\nabla\widetilde{f}$ in \eqref{definition of new f}, subtracting this identity from \eqref{limit eqn 2}  gives \eqref{limit ODE 0}.

Finally, although by now \eqref{limit ODE 0} is only understood in the distributional sense, by Lemma \ref{lem constancy} and our assumption on $\rho$, the right hand side of \eqref{limit ODE 0} is continuous in $t$. Then by standard regularity theory for ODEs, we deduce that $q_j$ is $C^1$ in $t$ and it satisfies \eqref{limit ODE 0} in the classical sense.
\end{proof}

\section{Blow-up limits: Proof of Theorem \ref{thm first time singularity}}\label{sec blow-up limits}
\setcounter{equation}{0}

Now we turn to the study of first time singularities and the proof of Theorem \ref{thm first time singularity}.

In this section, $u$ is a solution of \eqref{eqn} satisfying {\bf (H1-H3)}. The blow-up sequence is defined as
\[	u^\lambda(x,t):=\lambda^2 u(\lambda x,\lambda^2 t), \quad \nabla v^\lambda(x,t)=\lambda\nabla v(\lambda x,\lambda^2 t).\]
These functions are defined on $Q_{\lambda^{-1}}^-$, which tends to $\R^2\times(-\infty,0]$ as $\lambda\to0$.
By Theorem \ref{thm convergence} Item \ref{item:limit},  we may assume that for a subsequence (not relabelling) $\lambda\to0$,
\[u^\lambda(x,t)dx \rightharpoonup \mu_t, \quad \forall t\leq 0.\]

\subsection{Preliminary estimates}
First we notice the following two facts. Recall that Hypothesis {\bf(H3)} says that $u(x,t)$ forms a Dirac measure of mass $m$ at the origin as $t\to0^-$.
\begin{lem}\label{lem time 0}
  $\mu_0=m\delta_0$.
\end{lem}
\begin{proof}
For any $\lambda>0$, by {\bf(H3)}, as $t\to 0^-$,
\[ u^\lambda(x,t)dx \rightharpoonup m\delta_0+u_0^\lambda(x)dx, \quad \mbox{where} \quad  u_0^\lambda(x):=\lambda^2 u_0(\lambda x).\]
Because $u_0\in L^1(B_1)$, as $\lambda\to 0$,
\[ u_0^\lambda(x)dx\rightharpoonup 0.\]
Then by the uniform Lipschitz continuity of $\mathcal{L}_t^{\lambda_i}(\psi)$ (for any fixed $\psi\in C_0^2(\R^2)$, see \eqref{linear functional}), we deduce that
\[\lim_{t\to0}\lim_{\lambda\to 0}u^{\lambda}(x,t)dx=m\delta_0. \qedhere\]
\end{proof}

\begin{lem}\label{lem application of maximum principle}
  $\rho\equiv 0$ in $\left(\R^2\times(-\infty,0]\right)\setminus\Sigma$.
\end{lem}
\begin{proof}
	Denote 	$\left(\R^2\times(-\infty,0)\right)\setminus \Sigma:=\Omega$.
	In this open set, $\rho$ is  a classical solution of a linear parabolic equation
(i.e. the first equation in \eqref{eqn}). By Lemma \ref{lem time 0}, $\rho(y,0)=0$ in $\R^2\setminus\{0\}$.
The desired claim then follows from the strong maximum principle as in Step 1 in the proof of Lemma \ref{lem structure of backward weak sol}.
\end{proof}

By this lemma and Lemma \ref{lem constancy}, there exists an $N\in\mathbb{N}$ such that for any $t<0$,  $\Sigma_t$ consists of exactly $N$ distinct points, $q_1(t)$, $\dots$, $q_N(t)$, and
\[
\left\{
  \begin{array}{ll}
   \mu_t=8\pi\sum_{j=1}^{N}\delta_{q_j(t)},\\
    \nabla v_\infty(x,t)=-4\sum_{j=1}^{N}\frac{x-q_j(t)}{|x-q_j(t)|^2}.
 \end{array}
\right.\]
Furthermore, each $q_j(t)$ is $1/2$-H\"{o}lder in $t$ (by Proposition \ref{prop local structure}) and $q_j(0)=0$ (by Lemma \ref{lem time 0}).

As a consequence we obtain the quantization of mass, $m=8\pi N$. This also implies that, even though the blow-up limit may be not unique, the number of Dirac measures in every blow-up limit is the same $N$.

Because $\rho\equiv 0$, an application of Theorem \ref{thm limit eqn} gives
\begin{prop}\label{prop limiting ODE}
  For each $j$, $q_j(t)\in C((-\infty,0], \R^2)\cap C^1((-\infty,0), \R^2)$. Moreover,
  \begin{enumerate}
  	\item if $N=1$, then $q_1(t)\equiv 0$;
  	 \item if $N\geq 2$, then the vector valued function $(q_j(t))$ satisfies
  	 \begin{equation}\label{limiting ODE 1}
  	 	q_j^\prime(t)=4\sum\limits_{k\neq j}\dfrac{q_k(t)-q_j(t)}{|q_k(t)-q_j(t)|^2}, \quad  q_j(0)=0.
  	 \end{equation}
  \end{enumerate}
\end{prop}

To finish the proof of Theorem \ref{thm first time singularity}, it remains to show that
\begin{prop}\label{prop self-similarity}
  There exist $N$ distinct points $p_1$,$\dots$, $p_N\in\R^2$ such that
  \[ q_j(t)\equiv \sqrt{-t} p_j, \quad \forall t\leq 0.\]
  Moreover, $(p_j)$ is a critical point of the renormalized energy $\mathcal{W}$.
\end{prop}
If $N=1$, it has been established in Propositon \ref{prop limiting ODE} that $q_1(t)\equiv 0$, so we will assume $N\geq 2$ in the remaining part of this section.

We first prove several technical estimates on $q_j(t)$.
\begin{lem}
	For any $t<0$,
	\begin{equation}\label{total first momentum}
		\sum_{j=1}^{N}q_j(t)=0,
	\end{equation}
	\begin{equation}\label{total second momentum}
		\sum_{j=1}^{N}|q_j(t)|^2=2N(N-1)|t|.
	\end{equation}
\end{lem}
\begin{proof}
	By \eqref{limiting ODE 1}, we have
	\[\left\{
	\begin{array}{ll}
		\frac{d}{dt} \sum_{j=1}^{N}q_j(t)=4\sum_{j=1}^{N}\sum_{k\neq j}\frac{q_k(t)-q_j(t)}{|q_k(t)-q_j(t)|^2}=0, \\
		\sum_{j=1}^{N}q_j(0)=0.
	\end{array}
	\right.\\
	\]
	Then \eqref{total first momentum} follows by continuity.
	
	In the same way, we obtain
	\[\left\{
	\begin{array}{ll}
		\frac{d}{dt}\sum_{j=1}^{N}|q_j(t)|^2=8 \sum_{j=1}^{N} q_j(t)\cdot \sum_{k\neq j}\frac{q_k(t)-q_j(t)}{|q_k(t)-q_j(t)|^2}=2N(N-1), \\
		\sum_{j=1}^{N}|q_j(0)|^2=0.
	\end{array}
	\right.\\
	\]
	Then \eqref{total second momentum} follows by an integration in $t$.
\end{proof}
\begin{coro}\label{coro upper bound}
	For any $t<0$,
	\[\max_{1\leq j \leq N}|q_j(t)|\leq \sqrt{2N(N-1)}\sqrt{-t}.\]
\end{coro}
\begin{lem}\label{lem lower bound}
	There exists a constant $c_\ast>0$ such that for any $t<0$,
	\[ \min_{1\leq j\neq k \leq N}|q_j(t)-q_k(t)|\geq c_\ast\sqrt{-t}.\]
\end{lem}
\begin{proof}
Assume by the contrary, there exists a sequence of $t_i<0$, and two indicies $j\neq k$ such that
\begin{equation}\label{absurd assumption}
	\lim_{i\to+\infty}\frac{|q_j(t_i)-q_k(t_i)|}{\sqrt{-t_i}}=0.
\end{equation}
Set
\[\mu^i_t(A):= \mu_{|t_i|t}\left(\frac{A}{\sqrt{-t_i}}\right), \quad \forall A\subset \R^2 ~~ \text{Borel}.\]
As in the proof of Theorem \ref{thm convergence} Item \ref{item:limit}, after passing to a subsequence, we may assume $\mu^i_t\rightharpoonup \mu^\infty_t$ for any $t\leq 0$. By a diagonal argument as in the proof of Proposition \ref{prop multiplicity one} (see Step 1 therein), we deduce that
$\mu^\infty_t$ is also a blow-up limit of $u$ at $(0,0)$.

By the definition of $\mu^i_t$,
\begin{equation}\label{7.1}
	\mu_{-1}^i\geq 8\pi  \delta_{q_j(t_i)/\sqrt{-t_i}}+8\pi \delta_{q_k(t_i)/\sqrt{-t_i}}.
\end{equation}
By Corollary \ref{coro upper bound}, we may take a subsequence of $i$ so that both $q_j(t_i)/\sqrt{-t_i}$ and $q_k(t_i)/\sqrt{-t_i}$ converge.
By \eqref{absurd assumption},  their limit points coincide, which we denote by $q_\infty$. Passing to the limit in \eqref{7.1}, we obtain
\[\mu_{-1}^\infty \geq 16\pi \delta_{q_\infty}.\]
This is a contradiction with Proposition \ref{prop multiplicity one}.
\end{proof}

\subsection{The ODE analysis: Proof of Propostion \ref{prop self-similarity}}
In view of the above estimates, it is better to normalize the blow-up limits $(q_j(t))$.
\begin{defi}[Renormalized blow-up limits]\label{defi 8.8}
By setting $s:=-\log(-t)$, $s\in\R$, we define the renormalized blow-up limit to be
\[(p_j(s)):=\left(\frac{1}{\sqrt{-t}}q_j(t)\right).\]
\end{defi}

The following lemma gives a characterization of  the energy level of remormialized blow-up limits at infinite time, by using the renormalized energy $\mathcal{W}$ (see \eqref{renormalized energy} for definition).
\begin{lem}\label{lem critical energy levels}
	\begin{enumerate}
\item For any renormalized blow up limit $(p_j(s))$ constructed from a blow-up sequence at scales $\lambda_i\to0$, it is a solution to the ODE system
\begin{equation}\label{limiting ODE transformed}
	p_j^\prime(s)=\frac{1}{2}p_j(s)+4\sum_{k\neq j}\frac{p_k(s)-p_j(s)}{|p_k(s)-p_j(s)|^2}, \quad s\in\R.
\end{equation}
\item The two limits
\[\mathcal{W}_{\pm}(\{\lambda_i\}):= \lim_{s\to\pm\infty}\mathcal{W}\left(p_1(s),\cdots, p_N(s)\right)\]
exist, and they are   critical energy levels of $\mathcal{W}$.
\item There exists a constant $C_\ast$ depending only on $N$ and $c_\ast$ such that
\begin{equation}\label{energy level bound}
	-C_\ast\leq \mathcal{W}_-(\{\lambda_i\})\leq \mathcal{W}_+(\{\lambda_i\})\leq C_\ast.
\end{equation}
	\end{enumerate}
\end{lem}
\begin{proof}
(1) The ODE \eqref{limiting ODE transformed} follows by taking a change of variables in	
	\eqref{limiting ODE 1}. 
	
(2)	By Corollary \ref{coro upper bound},
	\begin{equation*}\label{upper bound}
		\max_{1\leq j \leq N}\sup_{s\in\R}|p_j(s)|\leq \sqrt{2N(N-1)}.
	\end{equation*}
	By Lemma \ref{lem lower bound},
	\begin{equation*}\label{lower bound}
		\min_{1\leq j \neq k\leq N}\inf_{s\in\R}|p_j(s)-p_k(s)|\geq c_\ast.
	\end{equation*}
Hence there exists a constant $C_\ast$  depending only on $N$ and $c_\ast$ such that
\begin{equation}\label{bound on energy}
	-C_\ast\leq \mathcal{W}(p_1(s),\cdots, p_N(s))\leq C_\ast.
\end{equation}
The ODE system \eqref{limiting ODE transformed} is the gradient flow of the function $\mathcal{W}$,
so  $\mathcal{W}(p_1(s),\cdots, p_N(s))$ is non-increasing in $s$. Thus
	\begin{equation}\label{energy limit}
\mathcal{W}_{\pm}(\{\lambda_i\}):=\lim_{s\to\pm\infty}\mathcal{W}(p_1(s),\cdots, p_N(s))
	\end{equation}
	are two well-defined, finite constants. This also implies that the $\omega$-limit points  and $\alpha$-limit points of $(p_j(s))$ are critical points of $\mathcal{W}$. They may not be unique, but the
	corresponding energy levels are $\mathcal{W}_\pm(\{\lambda_i\})$ respectively.
	
(3)	 Finally, \eqref{energy level bound} is a consequence of \eqref{bound on energy} and \eqref{energy limit}.
\end{proof}

In the following we denote, for any $s\in\R$,
\[ \overline{u}^s(y):=e^{-s} u\left(e^{-s/2} y, -e^{-s}\right),\]
which is just the slice of the blow-up sequence $u^{e^{-s/2}}$ at time $-1$.
Using $ \overline{u}^s$, all of the results established in the previous subsection can be summarized as follows.
\begin{lem}\label{lem N max points}
There exists a $T_\sharp>0$ such that for any $s>
T_\sharp$,  there exist $N$ points
$\widetilde{p}_1(s)$, $\dots$, $\widetilde{p}_N(s)$ satisfying
\begin{equation}\label{configuration good 0}
	\min_{j\neq k}|\widetilde{p}_j(s)-\widetilde{p}_k(s)|\geq c_\ast/2 \quad \mbox{and} \quad   \sum_{j=1}^{N} |\widetilde{p}_j(s)|^2\leq 2N^2,
\end{equation}
such that
$ \overline{u}^s(y)dy$
is weakly close to
$8\pi\sum_{j=1}^{N}\delta_{\widetilde{p}_j(s)}$
as Radon measures.
\end{lem}

Next we give a canonical construction of the  Dirac measures by using  an optimal approximation procedure.
For any $(p_1,\cdots, p_N)\in\R^{2N}$ satisfying ($c_\ast$ as in Lemma \ref{lem lower bound})
\[ \min_{j\neq k}|p_j-p_k|\geq c_\ast/4 \quad \text{and} \quad \sum_{j=1}^{N}|p_j|^2 \leq 2N^2+c_\ast^2,\]
define a positive smooth function $\omega(x;p_1,\cdots, p_N)$ such that it equals
$|x-p_j|^2$ in $B_{c_\ast/8}(p_j)$,  it is bounded below by $c_\ast^2/64$ outside $\cup_{j=1}^N B_{c_\ast/4}(p_j)$ and it is equal to a  positive constant outside $B_{2N+2c_\ast}$.

\begin{lem}\label{lem construction of canonical Dirac}
There exists a $T_\ast>0$ such that for any $s>T_\ast$, the following minimization problem
\[\min_{(p_1,\dots, p_N)\in\R^{2N}}\int_{B_{2N+3c_\ast}} \omega(x;p_1,\cdots, p_N)\overline{u}^s(x)dx\]
 has a unique minimizer $(p_{j,\ast}(s))$. Moreover,
 \begin{enumerate}
 \item for any $s>T_\ast$,
 \begin{equation}\label{configuration good}
 	 \min_{j\neq k}|p_{j,\ast}(s)-p_{k,\ast}(s)|\geq c_\ast/8 \quad \text{and} \quad \sum_{j=1}^{N}|p_{j,\ast}(s)|^2 \leq 2N^2+2c_\ast^2;
 \end{equation}
 \item for each $j$, $p_{j,\ast}\in C^1(T_\ast,+\infty)$.
 \end{enumerate}

\end{lem}
\begin{proof}
Denote
\[\mathcal{J}^s(p_1,\cdots, p_N):=\int_{B_{2N+3c_\ast}}\omega(x;p_1,\cdots, p_N)\overline{u}^s(x)dx.\]
It is smooth in $p_1$, ..., $p_N$ and positive everywhere.

By Lemma \ref{lem N max points},  we have
\begin{enumerate}
	\item $\lim_{s\to +\infty}\mathcal{J}^s(\widetilde{p}_1(s),\cdots, \widetilde{p}_N(s))=0$;
	\item  for any $\varepsilon>0$, there exist $T(\varepsilon)>0$ and $\delta(\varepsilon) >0$ such that, for any $s>T(\varepsilon)$, if $(p_j)$ satisfies
	$ \sum_{j=1}^{N}|p_j-\widetilde{p}_j(s)|>\delta(\varepsilon)$, then
	\[\mathcal{J}^s(p_1,\cdots, p_N)>\varepsilon.\]
\end{enumerate}
Thus there exists a $T_\ast>0$ such that for any $s\geq T_\ast$, the minima of $\mathcal{J}^s$ is attained at some point. Denote it by $(p_{j,\ast}(s))$. By the above Property (1), we have
\[\lim_{s\to +\infty}\sum_{j=1}^{N}|p_{j,\ast}(s)-\widetilde{p}_j(s)|=0.\]
Combining this relation with \eqref{configuration good 0}, we obtain \eqref{configuration good}, perhaps after taking a larger $T_\ast$.

By Lemma \ref{lem N max points},  there exists a small constant $\delta>0$ such that for any  $s>T_\ast$  and $(p_j)$ satisfying
$\sum_{j=1}^{N}|p_j-\widetilde{p}_j(t)|<\delta$,
it holds that
\begin{equation}\label{close to identity}
	 \sum_{j,k=1}^{N}\left| \int_{B_{2N+3c_\ast}} \frac{\partial^2\omega}{\partial p_j\partial p_k}(x;p_1,\cdots, p_N)\overline{u}^s(x)dx-\delta_{jk}\right|\leq \delta,
\end{equation}
where $\delta_{jk}$ is the Kronecker delta symbol.
This implies that $\mathcal{J}^s$ is strictly convex in a fixed neighborhood of $(\widetilde{p}_j(s))$. Therefore its minima point is unique. By the implicit function theorem, $p_{j,\ast}(s)$ is continuously differentiable in $s$.
\end{proof}

\begin{lem}\label{lem C1 convergence}
  For any sequence $T_i\to+\infty$, there exists a subsequence (not relabelling) such that $(p_{j,\ast}(s+T_i))$ converges to a renormalized blow-up limit in $C^1_{loc}(\R)$.
\end{lem}
\begin{proof}
Let $\lambda_i:=e^{-T_i/2}$.
Take a subsequence of $\lambda_i$  so that  the blow-up sequence $u^{\lambda_i}$ converges to a blow-up limit $8\pi\sum_{j=1}^N\delta_{q_j(t)}$ (in the sense of Theorem \ref{thm convergence}). Let $(p_j(s))$ be the corresponding renormalized blow up limit as defined in Definition \ref{defi 8.8}. By the definition of $\mathcal{J}^s$, we have
\[\lim_{T_i\to +\infty}\mathcal{J}^{s+T_i}\left(p_1(s),\cdots, p_N(s)\right)=0.\]
Then by the minimality of $(p_{j,\ast}(s))$ and the strict convexity of $\mathcal{J}^{s+T_i}$ near $\left(p_j(s)\right)$, we deduce that $(p_{j,\ast}(s+T_i))$  converges to $(p_j(s))$ uniformly on any compact set of $\R$.

Next we prove  the uniform convergence of $(p_{j,\ast}^\prime(s+T_i))$.
The proof of the previous lemma implies that for all $t<0$, the minimization problem
\[\int_{B_{2N+3c_\ast}}\omega(x;p_1,\cdots, p_N)u^{\lambda_i}(x,t)dx\]
has a unique minimizer $(p_j^{\lambda_i}(t))$. In fact, the uniqueness of minimizer implies that
\[\left( p_j^{\lambda_i}(t)\right)=\left(p_{j,\ast}(-\log(-t)+T_i)\right).\]
Hence it suffices to establish  the uniform convergence for the derivative of $p_j^{\lambda_i}$.

By the minimization condition for $(p_j^{\lambda_i}(t))$, we have
\[\int_{B_{2N+3c_\ast}}\frac{\partial\omega}{\partial p_j}(x;p_1^{\lambda_i}(t),\cdots, p_N^{\lambda_i}(t))u^{\lambda_i}(x,t)dx=0, \quad \forall j=1,\cdots, N.\]
Differentiating this identity in $t$, we obtain
\begin{eqnarray}\label{representation of time derivative}
	&&\sum_{k=1}^{N}\left(\int_{B_{2N+3c_\ast}}\frac{\partial^2\omega}{\partial p_j\partial p_k}(x;p_1^{\lambda_i}(t),\cdots, p_N^{\lambda_i}(t))u^{\lambda_i}(x,t)dx\right) \frac{dp_k^{\lambda_i}}{dt}(t) \nonumber\\
	&=& -\int_{B_{2N+3c_\ast}}\frac{\partial\omega}{\partial p_j}(x;p_1^{\lambda_i}(t),\cdots, p_N^{\lambda_i}(t))\partial_tu^{\lambda_i}(x,t)dx\\
	&=& -\int_{B_{2N+3c_\ast}}\Delta_x\frac{\partial\omega}{\partial p_j}(x;p_1^{\lambda_i}(t),\cdots, p_N^{\lambda_i}(t))u^{\lambda_i}(x,t)dx \nonumber\\
	&& +\frac{1}{4\pi} \int_{\R^2}\int_{\R^2}\Theta_{\frac{\partial\omega}{\partial p_j}}(x,y) u^{\lambda_i}(x,t)u^{\lambda_i}(y,t)dxdy. \nonumber
\end{eqnarray}
In the above, we have used Lemma \ref{lem symmetrization} and the fact that $\frac{\partial\omega}{\partial p_j}(x;p_1^{\lambda_i}(t),\cdots, p_N^{\lambda_i}(t))$ is compactly supported in $x$.
In view of \eqref{close to identity}, we can solve $dp_{k}^{\lambda_i}/dt$ from this equation.
By the convergence of  $u^{\lambda_i}(x,t)dx$, the last two integrals converge uniformly on any compact set of $(-\infty,0)$.  This then gives the uniform convergence of $dp_{k}^{\lambda_i}/dt$ on any compact set of $(-\infty,0)$, and consequently, the uniform convergence of $(p_{j,\ast}^\prime(s+T_i))$ on any compact set of $\R$.
\end{proof}

Let
\[E(s):=\mathcal{W}\left(p_{1,\ast}(s),\cdots, p_{N,\ast}(s)\right).\]
\begin{lem}
 The limit $ \lim_{s\to +\infty}E(s)$
exists, and it is a critical energy level of $\mathcal{W}$ in $[-C_\ast,C_\ast]$, where $C_\ast$ is the constant in \eqref{energy level bound}.
\end{lem}
\begin{proof}
%For any sequence $\lambda_i\to0$, there exists a subsequence (not relabelling) such that  $E(\lambda_i^2t)$ either converges to a strictly decreasing function or converges to a critical energy level of $\mathcal{W}$ between $-C_\ast$ and $C_\ast$.

First, by \eqref{configuration good} and the definition of $\mathcal{W}$,  $E(s)$ is a continuous function on $[T_\ast,+\infty)$. Moreover, for any $s\geq T_\ast$,
\begin{equation}\label{uniform bound on E}
	-C_\ast\leq E(s)\leq C_\ast.
\end{equation}

Next, let us recall a fact about the critical energy levels of $\mathcal{W}$. Because critical points of $\mathcal{W}$ are identical to those of $e^{-\mathcal{W}}$, and $e^{-\mathcal{W}}$
is a real analytic function (by the explicit form of $\mathcal{W}$  in \eqref{renormalized energy}), by the Lojasiewicz inequality for real analytic functions (see Lojasiewicz \cite{Lojasiewicz}), its critical  energy levels are discrete. Hence there are only finitely many  critical energy
levels of $\mathcal{W}$ in the interval $[-C_\ast,C_\ast]$. Denote them by $E_1, \cdots, E_K$ for some $K\in\mathbb{N}$. Let
\begin{equation}\label{gap between critical levels}
	\delta:=\min_{1\leq i\neq j\leq K}|E_i-E_j|>0
\end{equation}
be the minimal gap between these energy levels. Then for any non-static gradient flow $(p_j(s))$ of $\mathcal{W}$, if $\mathcal{W}(p_1(s),\cdots, p_N(s))\in[-C_\ast,C_\ast]$, we must have
\begin{equation}\label{energy gap for a gradient flow}
  \lim_{s\to-\infty}\mathcal{W}(p_1(s),\cdots, p_N(s))- \lim_{s\to+\infty}\mathcal{W}(p_1(s),\cdots, p_N(s))\geq \delta.
\end{equation}

For any $\varepsilon\in(0,\delta/10)$, set
\[\mathcal{I}_\varepsilon:=\left\{s: ~  s>T_\ast, ~ \min_{j=1,\cdots, K}|E(s)-E_j|>\varepsilon\right\}.\]
It is an open subset of $(T_\ast,+\infty)$. Assume
\[\mathcal{I}_{\varepsilon}=\cup_{\alpha}\mathcal{I}_{\varepsilon,\alpha},\]
where  $\mathcal{I}_{\varepsilon,\alpha}=(t_{\varepsilon,\alpha}^-,t_{\varepsilon,\alpha}^+)$ are the decomposition of $\mathcal{I}$ into pairwise disjoint, open intervals.  We claim that \\
{\bf Claim.} The number of these intervals is finite.

By Lemma \ref{lem C1 convergence}, for any $T>0$ and any $s\geq T_\ast+T$, $(p_{j,\ast})$ is close to   a gradient flow of $\mathcal{W}$ in $[s-T,s+T]$. Because in each $\mathcal{I}_{\varepsilon,\alpha}$, $E(s)$ is not close   to any critical energy level $E_j$, in every $\mathcal{I}_{\varepsilon,\alpha}$ (with at most finitely many exceptions),  $(p_{j,\ast})$ cannot be close to any static solution of \eqref{limiting ODE transformed} (i.e. a critical point of $\mathcal{W}$).  In view of \eqref{energy gap for a gradient flow},  this is possible only if (at least when $t_{\varepsilon,\alpha}^-$ is sufficiently large)
\begin{equation}\label{8.14}
	E(t_{\varepsilon,\alpha}^-)>E(t_{\varepsilon,\alpha}^+)+\delta-4\varepsilon.
\end{equation}
On the other hand, by the definition of $\mathcal{I}_\varepsilon$, we have
\begin{equation}\label{8.15}
\sup_{s_1,s_2\in[t_{\varepsilon,\alpha-1}^+,t_{\varepsilon,\alpha}^-]}|E(s_1)-E(s_2)|\leq 2\varepsilon.
\end{equation}
Combining \eqref{8.14} and \eqref{8.15}, we get a discrete monotonicity relation
\[E(t_{\varepsilon,\alpha-1}^+)-E(t_{\varepsilon,\alpha}^+)\geq \delta-6\varepsilon\geq \frac{\delta}{8}.\]
Combining this inequality with \eqref{uniform bound on E}, we deduce that there are only finitely many $\mathcal{I}_{\varepsilon,\alpha}$. This finishes the proof of this claim.

This claim can be reformulated as the statement that, for any $\varepsilon>0$, there exists a $T(\varepsilon)>0$ such that for any $s\geq T(\varepsilon)$,
\[\min_{j=1,\cdots, K}|E(s)-E_j|\leq \varepsilon.\]
By the gap between different critical energy levels in \eqref{gap between critical levels} and the continuity of $E(s)$,  we then deduce that there exists a fixed $j$ such that
 for any $\varepsilon\in(0,\delta/10)$ and $s\geq T(\varepsilon)$,
\[|E(s)-E_j|\leq \varepsilon. \qedhere\]
\end{proof}

\begin{proof}[Completion of the proof of Theorem \ref{thm first time singularity}]
By the previous lemma, there exists an energy critical level $E$ of $\mathcal{W}$ such that, for any renormalized blow-up limit $(p_j(s))$,
\[\mathcal{W}(p_1(s),\cdots, p_N(s))\equiv E.\]
Because $(p_j(s))$ is the gradient flow of $\mathcal{W}$, this is possible only if it does not depend on $s$. Hence it must be a critical point of $\mathcal{W}$.
\end{proof}
\begin{rmk}
	Because blow up limits are obtained by a compactness argument and critical points of $\mathcal{W}$ are not discrete, it is not clear if the renormalized blow-up limit is unique.
\end{rmk}

\section{Entire solutions}\label{sec entire solutions}
\setcounter{equation}{0}

In  this section we prove Theorem \ref{thm entire sol}. We will also explain, in the setting of Section \ref{sec convergence theory}, how to obtain entire solutions from suitable rescalings around a blow up point.

\begin{proof}[{Proof of Theorem \ref{thm entire sol}}]
If $u$ is an entire solution, by Theorem \ref{main result}, the   blow-down  sequence converges to a limit $(\mu_t)$ in $\R^2\times\R$.
By Theorem \ref{thm ancient sol}, there exists an $N\in\mathbb{N}$ such that $\mu_0=8\pi N$. On the other hand, because the blow-down sequence converges  in $Q_1$,  an application of Proposition \ref{prop multiplicity one} at $t=0$ shows that $N=1$.   Then by Lemma \ref{lem structure of backward weak sol} and Lemma \ref{lem structure of forward weak sol},  we get $\mu_t=8\pi\delta_0$ for $t<0$ and $t>0$ respectively.
\end{proof}

Next,  we use the convergence theory in Section \ref{sec convergence theory} to show how entire solutions arise as micro-models of singularity formations.

We work
under the same assumptions of Theorem \ref{thm convergence}. After localization (see Lemma \ref{lem localization}), we may assume that
\begin{equation}\label{one bubble assumption}
	u_i(x,0)dx \rightharpoonup 8\pi\delta_{0}+\rho(x,0)dx \quad \mbox{in} ~~ B_1.
\end{equation}
This implies that
\[\max_{B_1}u_i(x,0) \to+\infty.\]
 Because $u_i$ converges   to $\rho$ in $C_{loc}(B_1\setminus\{0\})$,  this maximum is attained at an interior point, say $x_i$, where
\[|x_i|\to0.\]

Denote
$R_i:=u_i(x_i,0)^{1/2}$.
Set
$\widetilde{u}_i(x,t):=R_i^{-2}u_i(x_i+R_i^{-1} x, R_i^{-2}t )$
and define $\nabla\widetilde{v}_i$ as in  \eqref{representation for v} by using $\widetilde{u}_i$. Then we have
\begin{thm}\label{thm construction of entire sol}
	As $i\to+\infty$,  $\widetilde{u}_i$ converges to a limit $\widetilde{u}_\infty$ in $C^{2,1}_{loc}(\R^2\times \R)$, where $\widetilde{u}_\infty$ is an entire solution of \eqref{eqn}.
\end{thm}
\begin{proof}
	By Theorem \ref{thm convergence}, for each $t\in\R$, there exist finitely many points $q_j(t)$, $j=1$, ..., $N(t)$, such that	as Radon mearures on $\R^2$,
	\begin{equation}\label{limit of entire sol}
		\widetilde{u}_i(x,t)dx\rightharpoonup 8\pi\sum_{j=1}^{N(t)}\delta_{q_j(t)}+\widetilde{u}_\infty(x,t)dx.
	\end{equation}

	By \eqref{one bubble assumption} and Lemma \ref{lem symmetrization}, for any $\varepsilon>0$, there exists a $\delta>0$ such that for any $t\in(-\delta,\delta)$,
	\[\limsup_{i\to+\infty}\int_{B_\delta}u_i(x,t)dx < 8\pi+\varepsilon.\]	
	After a scaling, this is transformed into
	\begin{equation}\label{8.1}
\limsup_{i\to+\infty}\int_{B_{\delta R_i}}\widetilde{u}_i(x,t)dx \leq 8\pi+\varepsilon, \quad\forall t\in(-\delta R_i^2, \delta R_i^2).
\end{equation}
	Combining this upper bound with \eqref{limit of entire sol}, we see  $N(t)\leq 1$. Moreover, by Theorem \ref{thm convergence} Item \ref{item:support} and Proposition \ref{prop multiplicity one},
	\begin{itemize}
		\item {\bf Alternative I:} either $N(t)=1$ and $\widetilde{u}_\infty(t)\equiv 0$ ;
		\item {\bf Alternative II:} or $N(t)=0$ and  $\widetilde{u}_\infty(t)\in C^2(\R^2)$ .
	\end{itemize}
	Let $\mathcal{I}$ be the set of those $t$ satisfying Alternative I.  By the same strong maximum principle argument used in the proof of Lemma \ref{lem structure of backward weak sol} we deduce that there exists a $T\leq +\infty$ such that
	\[\mathcal{I}=(-\infty,T].\]
	On the other hand, by definition, we have
	\[\max_{B_{R_i}}\widetilde{u}_i(x,0)=\widetilde{u}_i(0,0)=1.\]
	Then by Theorem \ref{thm ep regularity} (applied to the cylinder  $Q_r(x,0)$,  for any $x\in\R^2$ and a fixed, sufficiently small $r$),  there exists a $\delta>0$ such that $\widetilde{u}_i\in C^{2,1}(\R^2\times(-\delta,\delta))$. As a consequence, $0\notin\mathcal{I}$, or equivalently, $T<0$.
	
If $T>-\infty$,	  $\widetilde{u}_\infty$ is smooth in $\R^2\times(T,+\infty)$ and at time $T$ it is a Dirac measure with mass $8\pi$. Then the same proof of Lemma \ref{lem structure of forward weak sol} (more precisely, the proof of \eqref{5.2}) leads to a contradiction. This contradiction implies that $T=-\infty$, that is, Alternative II holds for all $t$. Then by Theorem \ref{thm convergence} Item \ref{item: smooth outside}, $\widetilde{u}_i$ converges uniformly to  $\widetilde{u}_\infty$ on any compact set of $\R^2\times\R$. By standard parabolic regularity theory, we deduce that $\widetilde{u}_\infty$, with the corresponding $\nabla\widetilde{v}_\infty$,  is a classical solution of \eqref{eqn} on $\R^2\times\R$, i.e. it is an entire solution. Moreover, by passing to the limit in \eqref{8.1}, we see $\widetilde{u}_\infty$ satisfies the finite mass condition \eqref{finite mass condition}.
\end{proof}

\section{Boundary blow up points}\label{sec boundary blow up}
\setcounter{equation}{0}

This last section is intended as a remark on boundary blow up points.
Let $\Omega\subset \R^2$ be a smooth domain. Consider the initial-boundary value problem	
\[
\left\{\begin{aligned}
	&	u_t=\Delta u-\mbox{div}(u\nabla v),  \quad &\mbox{in} ~~ \Omega\times(0,T),\\
	&	-\Delta v+v=u,  \quad  &\mbox{in} ~~ \Omega\times(0,T), \\
	&	\partial_\nu u=\partial_\nu v=0,   \quad &\mbox{on} ~~ \partial\Omega\times(0,T),
\end{aligned}\right.
\]
where $\nu$ is the outward unit normal vector of $\partial\Omega$.

Solutions to this problem could also blow up in finite time, which is still caused by the concentration of $u$, see \cite[Chapter 11]{Suzuki1}. The blow up points could lie on the boundary.
The blow up analysis developed in this paper  can also be used to study boundary blow up points. More precisely, similar to Theorem \ref{thm first time singularity}, we have
\begin{thm}\label{thm boundary singularity}
	Assume $T$ is the first blow up time, and $a\in\partial\Omega$ is a boundary blow up point. Let
	\[u^\lambda(x,t):=\lambda^2u(a+\lambda x, T+\lambda^2 t), \quad \lambda \to 0.\]
 For any sequence $\lambda_i\to0$, there exists a subsequence (not relabelling), an open half space $H$ and finitely many points $p_j\in H$ and $q_k\in\partial H$ such that, for any $t<0$,
		\begin{equation}\label{10.1}
			u^{\lambda_i}(x,t)dx \rightharpoonup 8\pi\sum_{j}\delta_{\sqrt{-t}p_j} +4\pi\sum_{k}\delta_{\sqrt{-t}q_k} \quad \mbox{weakly as Radon measures}.
		\end{equation}
\end{thm}
We only briefly explain the proof of Theorem \ref{thm boundary singularity}. We can still work in the local setting as in Theorem \ref{thm first time singularity}: first take a small ball $B_r(a)$ around $a$, then take a diffeomorphism $\Omega\cap B_r(a)\mapsto H$ to flatten the boundary, where $H$ is an open half space. After scaling the radius of this ball to be $1$, extend $u$ and $v$ evenly to the whole $B_1$. Denote  by $g$ the Riemannian metric  obtained by pushing forward the original Euclidean metric through these transformations.
Then   we are in the following local setting:
\begin{itemize}
	\item $\widetilde{u} \in C^\infty(\overline{Q_1^-}\setminus\{(0,0)\})$, $\widetilde{u}>0$ and
	\[
	\sup_{t\in(-1,0)}\int_{B_1}\widetilde{u}(x,t)dx\leq M;
	\]
	\item $\widetilde{u}$ satisfies
	\[\widetilde{u}_t-\Delta_g\widetilde{u}=-\mbox{div}_g\left(\widetilde{u}\nabla_g \widetilde{v}+\nabla_g f\right) \quad \mbox{in} ~~ Q_1^-,\]
	where $\Delta_g$ is the Beltrami-Laplace operator with respect to the Riemannian metric $g$, $\text{div}_g$ is the corresponding divergence operator, $f$ is a smooth function in $Q_1$, $\nabla \widetilde{v}$ is given by
	\[
	\nabla \widetilde{v}(x,t)=\int_{B_1}\left[-\frac{1}{2\pi}\frac{x-y}{|x-y|^2}+\nabla R(x,y)\right]\widetilde{u}(y,t)dt, \quad \forall (x,t)\in Q_1^-,
	\]
	with $R$ a smooth function of $(x,y)$ (the regular part of the Green function);
	\item $\widetilde{u}$ is even symmetric with respect to the hyperplane $\partial H$; 
	\item there exists a nonnegative function $\widetilde{u}_0\in L^1(B_1)$ and a positive constant $m$ such that as $t\to0^-$,
	\[ \widetilde{u}(x,t)dx\rightharpoonup \widetilde{u}_0(x)dx+m\delta_0 \quad \text{weakly as Radon measures}.\]
\end{itemize}
After blow-up, the inhomogeneity caused by the Riemannian metric $g$ will disappear. The remaining proof is exactly the same as the one of Theorem \ref{thm first time singularity}. There is only one point which is different from the interior case: first, an application of Theorem \ref{thm first time singularity} gives the weak convergence
	\[\widetilde{u}^{\lambda_i}(x,t)dx \rightharpoonup 8\pi\sum_{j}\left(\delta_{\sqrt{-t}p_j}+\delta_{-\sqrt{-t}p_j}\right) +8\pi\sum_{k}\delta_{\sqrt{-t}q_k},\]
where $p_j\in H$, $q_k\in\partial H$, and all of these points are distinct from each other; next, because $\widetilde{u}$ is the even extension of $u$, restricting the above weak convergence to $\overline{H}$,  we get the weak converengece of $u^{\lambda_i}$ as stated in  \eqref{10.1}.

\appendix
\section{Proof of Theorem \ref{thm ep regularity}}\label{sec appendix}
\setcounter{equation}{0}

In this appendix we prove
 Theorem \ref{thm ep regularity}. It is a consequence of the following theorem.
 \begin{thm}\label{thm ep regularity 2}
Suppose $(u,v)$ is a classical solution of \eqref{eqn modified} in $Q_1$, satisfying
\begin{equation}\label{uniform mass bound}
  \sup_{|t|<1}\int_{B_1}u(x,t)dx<2\varepsilon_\ast,
\end{equation}
then
\begin{equation}\label{interior regular}
\|u\|_{C^{1+\alpha,(1+\alpha)/2}(Q_{1/2})}\leq C.
\end{equation}
 \end{thm}
The only difference  with Theorem \ref{thm ep regularity} is that now we make a stronger assumption \eqref{uniform mass bound}, which requires a small mass bound for all $|t|<1$. However, \eqref{uniform mass bound} follows by combining \eqref{small mass condition} with Lemma \ref{lem symmetrization}, perhaps after restricting to a smaller cylinder (e.g. $Q_{2\theta_\ast}$ for some small $\theta_\ast>0$) and then scaling this cylinder to the unit one.

Before proving this theorem, we  recall several inequalities. The first two inequalities are taken from  \cite[Lemma 4.2 and Lemma 11.1]{Suzuki1}.
\begin{lem}\label{lem A.2}
For any $\psi\in C_0^\infty(B_1)$, the following inequalities hold for any $s>1$, where $C>0$ is a constant determined by $\psi$ only:
\begin{equation}\label{L2 Sobolev}
\int_{B_1}u^2\psi dx\leq C\|u\|_1\left(\int_{B_1} u^{-1}|\nabla u|^2\psi dx\right)+C\|u\|_1^2,
\end{equation}
\begin{equation}\label{L3 Sobolev}
\int_{B_1}u^3\psi dx\leq \dfrac{C}{\log s}\left(\int_{B_1}(u \log u+e^{-1})dx\right)\left(\int_{B_1} |\nabla u|^2\psi dx\right)+C\|u\|_{1}^3+10s^3.
\end{equation}
\end{lem}

The next two inequalities are similar to  the corresponding ones for \eqref{eqn} (see \cite[Eqns. (11.15) and (11.16)]{Suzuki1}). Although there are additional terms $\nabla f$ and $g$ in \eqref{eqn modified}, due to their lower order nature, they only produce terms
which can be incorporated into the constant term.
\begin{lem}
For any $\psi\in C_0^2(B_1)$,
  \begin{equation}\label{entropy derivative}
\dfrac{d}{dt}\int_{B_1}(u\log u)\psi dx\leq -\dfrac{1}{2}\int_{B_1}u^{-1}|\nabla u|^2\psi dx+2\int_{B_1}u^{2} \psi dx+C.
\end{equation}

\begin{equation}\label{L2 derivative}
\frac{d}{d t} \int_{B_1} u^{2} \psi d x \leq -\int_{B_1}|\nabla u|^{2} \psi  d x +3\int_{B_1} u^{3} \psi d x+C.
 \end{equation}
\end{lem}

\begin{proof}[Proof of Theorem \ref{thm ep regularity 2}]
Take two functions $\psi_1\in C_0^\infty(B_1),\psi_2\in C_0^\infty(B_{3/4})$ such that $0\leq \psi_1, \psi_2\leq 1$, $\psi_1\equiv 1$ in $B_{3/4}$ and $\psi_2\equiv 1$ in $B_{2/3}$. Denote
\[ I_1(t):=\int_{B_1}\left(u\log u+e^{-1}\right)\psi_1 dx, \quad  J_1(t):=\int_{B_1}u^2\psi_1 dx\]
and
\[ I_2(t):=\int_{B_1}u^2\psi_2 dx, \quad  J_2(t):=\int_{B_1}u^3\psi_2 dx.\]

{\bf Step 1.}
 In view of \eqref{uniform mass bound}, combining \eqref{entropy derivative} and \eqref{L2 Sobolev} gives
\[ I_1^\prime(t)\leq -cJ_1(t)+C.\]
By the H\"{o}lder inequality, we have
\[ I_1(t)\leq CJ_1(t)^{\frac{1}{3}}+C.\]
Hence we get
\[ I_1^\prime(t)\leq -cI_1(t)^3+C.\]
Because $I_1(t)\geq 0$ for any $t\in(-1,1)$, this differential inequality implies that
\begin{equation}\label{uniform bound on J1}
  I_1(t)\leq C, \quad \forall t\in[-7/8,7/8].
\end{equation}

{\bf Step 2.}
 In view of \eqref{uniform bound on J1}, combining  \eqref{L3 Sobolev} and \eqref{L2 derivative} (where we choose a sufficiently large $s$, determined by the upper bound of $I_1$ in \eqref{uniform bound on J1}) gives
\[ I_2^\prime(t)\leq -cJ_2(t)+C.\]
By the H\"{o}lder inequality, we have
\[ I_2(t)\leq CJ_2(t)^{\frac{2}{3}}+C.\]
Hence we get
\[ I_2^\prime(t)\leq -cI_2(t)^{\frac{3}{2}}+C.\]
Because $I_2(t)\geq 0$ for any $t\in(-7/8,7/8)$, this differential inequality implies that
\begin{equation}\label{uniform bound on J2}
  I_2(t)\leq C, \quad \forall t\in[-3/4,3/4].
\end{equation}

{\bf Step 3.} We have shown that $u\in L^\infty(-3/4,3/4;L^2(B_{3/4}))$. Then by standard $W^{2,p}$ estimate and Sobolev embedding theorem, $\nabla v\in L^\infty(-3/4,3/4;L^3(B_{2/3}))$.  This implies that $u\nabla v\in L^\infty(-3/4,3/4;L^{6/5}(B_{3/4}))$.
Then  we can lift the regularity of $u$ and $v$ by bootstrapping  standard $W^{2,p}$ estimate  and Schauder estimate (see \cite[Theorem 4.8]{Lieberman}).
\end{proof}

\end{document}